\documentclass[10pt, a4paper, notitlepage]{article}
\pdfoutput=1

\usepackage{biblatex}
\usepackage{bbm}

\usepackage{etoolbox}
\usepackage{tikz}
\usetikzlibrary{calc}
\usetikzlibrary{cd}
\usetikzlibrary{decorations.markings}
\usetikzlibrary{decorations.pathreplacing}
\usetikzlibrary{decorations.pathmorphing}
\usetikzlibrary{decorations.text}
\usetikzlibrary{arrows.meta}
\usetikzlibrary{arrows}
\usetikzlibrary{positioning}
\usetikzlibrary{decorations.pathreplacing,fit}

\usepackage{geometry} 
\usepackage{tocloft} 
\usepackage{fancyhdr} 
\usepackage{longtable}
\usepackage{subcaption}
\usepackage{enumitem}
\usepackage{amssymb}
\usepackage{amsmath}
\usepackage{stackrel} 
\usepackage{uniinput}
\usepackage{amsthm}
\usepackage{stmaryrd}
\usepackage{mathtools}
\usepackage{fouridx}
\usepackage{xparse}
\usepackage{bm}
\usepackage{aliascnt}
\usepackage{import}

\usepackage{hyperref}
\theoremstyle{definition}
\newtheorem{theorem}{Theorem}
\let\emph\textbf

\newaliascnt{remark}{theorem}
\newaliascnt{definition}{theorem}
\newaliascnt{proposition}{theorem}
\newaliascnt{lemma}{theorem}
\newaliascnt{corollary}{theorem}
\newaliascnt{example}{theorem}
\newaliascnt{convention}{theorem}
\newaliascnt{todo}{theorem}

\newtheorem{remark}[remark]{Remark}
\newtheorem{definition}[definition]{Definition}
\newtheorem{proposition}[proposition]{Proposition}
\newtheorem{lemma}[lemma]{Lemma}
\newtheorem{corollary}[corollary]{Corollary}
\newtheorem{example}[example]{Example}
\newtheorem{convention}[convention]{Convention}

\aliascntresetthe{remark}
\aliascntresetthe{definition}
\aliascntresetthe{proposition}
\aliascntresetthe{lemma}
\aliascntresetthe{corollary}
\aliascntresetthe{example}
\aliascntresetthe{convention}
\aliascntresetthe{todo}

\numberwithin{equation}{section}
\numberwithin{figure}{section}
\numberwithin{table}{section}
\makeatletter\let\c@table\c@figure\makeatother

\numberwithin{theorem}{section}
\numberwithin{remark}{section}
\numberwithin{definition}{section}
\numberwithin{proposition}{section}
\numberwithin{lemma}{section}
\numberwithin{corollary}{section}
\numberwithin{example}{section}
\numberwithin{convention}{section}
\numberwithin{todo}{section}

\newcommand{\tinymath}[1]{\scalebox{0.7}{#1}}

\newcommand{\Id}{\operatorname{Id}}
\newcommand{\id}{\mathrm{id}}

\newcommand{\Hom}{\operatorname{Hom}}
\newcommand{\Ker}{\operatorname{Ker}}

\newcommand{\Fuk}{\operatorname{Fuk}}

\newcommand{\vspan}{\operatorname{span}}

\newcommand{\isoto}{\xrightarrow{\sim}}

\newcommand{\embeds}{\hookrightarrow}
\def\Im{\operatorname{Im}}

\newcommand{\Perf}{\operatorname{Perf}}

\newcommand{\M}{\mathcal{M}}

\newcommand{\Sym}{\operatorname{Sym}}

\newcommand{\Cat}{\operatorname{Cat}}

\newcommand{\restr}{|}

\renewcommand{\Re}{\operatorname{Re}}

\newcommand{\Hilbhor}{\operatorname{Hilb}^{\operatorname{hor}}}
\newcommand{\NH}{\operatorname{NH}}

\newcommand{\Res}{\operatorname{Res}}

\newcommand{\Comp}{\operatorname{Comp}}
\newcommand{\BC}{\operatorname{BC}}
\newcommand{\Map}{\operatorname{Map}}
\newcommand{\fib}{\operatorname{fib}}
\newcommand{\Fun}{\operatorname{Fun}}
\newcommand{\Ind}{\operatorname{Ind}}
\newcommand{\NHS}{\mathbf{NH}}
\newcommand{\St}{\operatorname{St}}
\newcommand{\Anycross}{\operatorname{Anycross}}
\newcommand{\Mincross}{\operatorname{Mincross}}

\newcommand{\KLRW}{\operatorname{KLRW}}

\newcommand{\shufflebox}[4]{%
  \begin{scope}[shift={#1}]%
    \draw (0,0) rectangle ($ 0.5*(#2,1) + 2*(#4, 0) $);%
    \foreach \x in {1,...,#2}{%
      \coordinate (#3-B\x) at ($ (#4, 0) + 0.5*(\x,0)+0.5*(-.5,0)$);%
      \coordinate (#3-T\x) at ($ (#4, 0) + 0.5*(\x,1)+0.5*(-.5,0)$);%
    }%
    \coordinate (#3-C) at ($ 0.25*(#2, 1) + (#4, 0) $);
  \end{scope}%
}

\newcommand{\permutbox}[6]{%
  \begin{scope}[shift={#1}, scale=0.5]
    \draw (0,0) -- (2,0)
      coordinate[pos=0.1] (bot-1)
      coordinate[pos=0.3] (bot-2)
      coordinate[pos=0.5] (bot-3)
      coordinate[pos=0.7] (bot-4)
      coordinate[pos=0.9] (bot-5);
    \draw (0,1) -- (2,1)
      coordinate[pos=0.1] (top-1)
      coordinate[pos=0.3] (top-2)
      coordinate[pos=0.5] (top-3)
      coordinate[pos=0.7] (top-4)
      coordinate[pos=0.9] (top-5);
    \draw (bot-1) -- (top-#2);
    \draw (bot-2) -- (top-#3);
    \draw (bot-3) -- (top-#4);
    \draw (bot-4) -- (top-#5);
    \draw (bot-5) -- (top-#6);
  \end{scope}%
}

\newcommand{\permgridrowwidth}{4} 

\newcommand{\permcell}[6]{%
  \pgfmathtruncatemacro{\col}{mod(#1,\permgridrowwidth)}%
  \pgfmathtruncatemacro{\row}{int(#1/\permgridrowwidth)}%
  \permutbox{(\col*1.5, -\row*1)}{#2}{#3}{#4}{#5}{#6}%
}

\title{The NilHecke schober}
\author{Jasper van de Kreeke}
\date{}

\begin{document}

\maketitle

\begin{abstract}
Perverse schobers can be used to describe Fukaya categories but are hard to axiomatize and construct. In this paper, we give an explicit construction of a perverse schober intended to accurately describe the Fukaya category of the horizontal Hilbert scheme considered by Aganagic et al.~in the frame of the knot invariant categorification program. The formalism is based on the notion of $ A_n $-schobers of Dyckerhoff and Wedrich. The construction is entirely algebraic and we check the schober axioms with the help of diagrammatic calculus.
\end{abstract}

\tableofcontents

\section{Introduction}
Perverse schobers are a concept of “sheaf of categories” over a given stratified base space $ B $. In the case of the symmetric product $ B = \Sym^{n+1} (ℂ) $, a full definition of perverse schober was achieved by Dyckerhoff-Wedrich \cite{Dyckerhoff-Wedrich}. It entails the datum of a cubical diagram of stable $ ∞ $-categories satisfying certain axioms:
\begin{equation*}
χ: \{0, 1\}^n → \St.
\end{equation*}
The horizontal Hilbert scheme $ \Hilbhor_n (ℂ^2) $ is a Liouville manifold and its Fukaya category with potential was studied in \cite{ADLSZ}. Up to generation questions, this category is entirely described as $ \Perf \NH_n $ where $ \NH_n $ is the NilHecke algebra on $ n $ strands. The perfect complexes of NilHecke algebras assemble into a cubical diagram as well:
\begin{align*}
\NHS_{n+1}: \{0, 1\}^n ≅ \Comp(n+1) &→ \St, \\
(σ_1, …, σ_k) &↦ \Perf (\NH_{σ_1} ¤_{ℂ⟦ħ⟧} … ¤_{ℂ⟦ħ⟧} \NH_{σ_k}).
\end{align*}
In this paper, we check that $ \NHS_{n+1} $ is a perverse schober on $ \Sym^{n+1} (ℂ) $ in the sense of Dyckerhoff-Wedrich.

\subsection*{Perverse schobers}
The concept of perverse schobers was introduced by Kapranov and Schechtman \cite{Kapranov-Schechtman} as the closest approximation to the notion of perverse sheaf of categories with respect to a given stratification. Their primary use case has been to describe Fukaya categories and similar categories in a way that allows for $ ∞ $-theoretical local constructions, effectively making the treatment more ergonomic compared to global methods.

The precise notion of perverse schober depends on the given base space and its stratification and is often hard to axiomatize. The reason is that a perverse schober on a given base should precisely decatigorify, upon taking $ K_0 $, to the corresponding notion of perverse sheaf on the same base space. Accurate notions of perverse schobers have been given so far for marked disks \cite{Kapranov-Schechtman}, marked surfaces \cite{Christ} and symmetric powers $ \Sym^{k+1} (ℂ) $ \cite{Dyckerhoff-Wedrich}.

\subsection*{Fukaya categories of Coulomb branches}
Fukaya categories of Coulomb branches have been studied in \cite{ADLSZ}, as part of the program to categorify knot invariants via mirror symmetry \cite{Aganagic-I, Aganagic-II}. The result of that paper is an explicit description of the subcategory of $ \Fuk(\M^{×} (Q, d), W) $ given by a certain set of objects $ T_θ $. More precisely, there is a matching of objects $ T_θ $ with red and black dot distributions on the line, and a matching of their hom spaces with linear spans of strand diagrams between two configurations. This amounts to the inclusion $ \Perf(\KLRW) ⊂ \Fuk(\M^{×} (Q, d), W) $.

Let us focus on the case of $ (Q, d) $ being the $ A_1 $-quiver without framing and with arbitrary dimension $ d = n+1 $. In this case, we have $ \M^{×} (Q, d) = \Hilbhor_{n+1} (ℂ^2) $ and $ \KLRW = \NH_{n+1} $. The horizontal Hilbert scheme comes with a map $ π: \Hilbhor_{n+1} (ℂ^2) → \Sym^{n+1} (ℂ) $, which is a fiber bundle over the open locus $ \Sym_{≠}^{n+1} (ℂ) $ consisting of points with pairwise distinct coordinates. This suggests that $ \Fuk(\Hilbhor_{n+1} (ℂ^2), W) $ is the global sections category of a “sheaf of categories” over $ \Sym^{n+1} (ℂ) $.

The schober perspective makes for the hypothesis that \cite{ADLSZ} is in fact the shadow of an equality between two different schobers: the “algebraic” NilHecke schober $ \NHS_{n+1} $ and a “geometric” horizontal Hilbert scheme schober $ F_x ≔ \Fuk(π^{-1} (x), W \restr_{π^{-1} (x)}) $. The present paper is a first step in that direction.

\subsection*{Structure of the paper}
In \autoref{sec:prelim}, we recall the definition of $ A_n $-schobers according to Dyckerhoff-Wedrich \cite{Dyckerhoff-Wedrich}. In \autoref{sec:nh}, we introduce the NilHecke schober $ \NHS_{n+1} $ and check that it satisfies the $ A_n $-schober axioms. In \autoref{sec:examples}, we demonstrate the verification of the schober axioms in two specific examples cases. In \autoref{sec:appchecks}, we sketch certain cases of the proof which have been left out of the main discussion.

\paragraph*{Acknowledgments}
The author would like to thank Mina Aganagic and Peng Zhou for supervision. The author was partially supported by
NWO Rubicon grant 019.232EN.029.

\section{Preliminaries on $ A_n $-schobers}
\label{sec:prelim}
In this section, we recall the definition of $ A_n $-schobers according to Dyckerhoff-Wedrich \cite{Dyckerhoff-Wedrich}. We start with the cubical indexing set $ \{0, 1\}^n $ and its bijective correspondence with compositions of $ n+1 $. After that, we provide explicit descriptions of bifactorization cubes and Beck-Chevalley cubes.

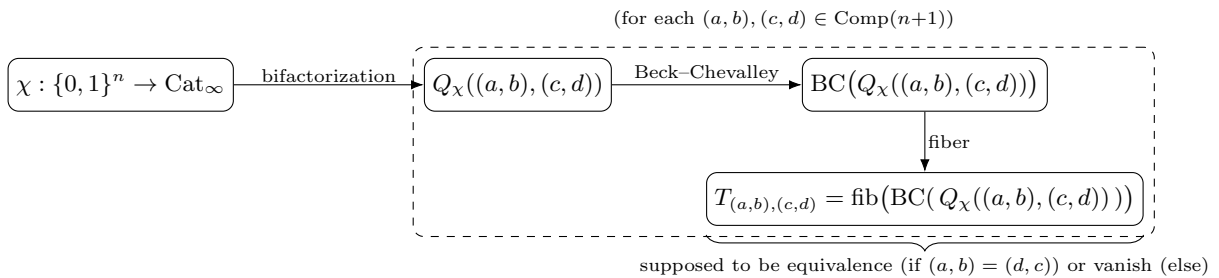
\begin{center}
\begin{tikzpicture}[
  >=Latex, font=\small,
  x=1mm, y=1mm,
  block/.style={draw, rounded corners, inner sep=3pt, minimum height=7mm, align=center},
  note/.style={font=\scriptsize, inner sep=1pt},
  dashedbox/.style={draw, dashed, rounded corners, inner sep=4pt},
  node distance=15mm
]

\node[block] (chi) {$\chi:\{0,1\}^n \to \Cat_∞ $};

\node[block, right=25mm of chi] (Q) {$Q_{\chi}((a, b),(c, d))$};

\node[block, right=25mm of Q] (BC) {$\mathrm{BC}\big(Q_{\chi}((a, b), (c, d))\big)$};
\node[block, below=8mm of BC] (T) {$T_{(a, b),(c, d)}=\fib\!\big(\mathrm{BC}(\,Q_{\chi}((a, b),(c, d))\,)\big)$};

\draw[->] (chi) -- (Q) node[midway, above, note] {bifactorization};
\draw[->] (Q) -- (BC) node[midway, above, note] {Beck--Chevalley};
\draw[->] (BC) -- (T)  node[midway, right, note] {fiber};

\node[dashedbox, fit=(Q)(BC)(T)] (foreachbox) {};
\node[note, above=2mm of foreachbox] {(for each $ (a, b), (c, d) \in \Comp(n{+}1)$)};

\draw [decorate, decoration={brace, mirror, amplitude=6pt}]
  ($(T.south west)+(0,-1.2)$) -- ($(T.south east)+(0,-1.2)$)
  node[midway, below=7pt, note] {supposed to be equivalence (if $ (a, b) = (d, c) $) or vanish (else)};

\end{tikzpicture}
\end{center}

The construction of bifactorization cubes and Beck-Chevalley cubes in principle applies to any given cubical diagram (of categories or parallel functors, respectively). However, the definition of $ A_n $-schobers references a chained and very specific subset of these constructions. The definition starts from the datum of one cubical diagram $ χ: \{0, 1\}^n → \Cat_∞ $. Consequently, for each pair of compositions $ (a, b), (c, d) ∈ \Comp(n+1) $ one constructs the bifactorization cube $ Q_χ((a, b), (c, d)) $. Then one constructs the Beck-Chevalley cube $ \BC(Q_χ ((a, b), (c, d))) $. Finally, one takes the fiber $ T_{(a, b), (c, d)} ≔ \fib(\BC(Q_χ ((a, b), (c, d)))) $. The definition of $ A_n $-schobers entails that these fibers $ T_{(a, b), (c, d)} $ for every $ (a, b), (c, d) ∈ \Comp(n+1) $ satisfy certain conditions.

In \autoref{sec:prelim-cube}, we recall notation for cubical diagrams. In \autoref{sec:prelim-bifcube}, we recall the notion of bifactorization cubes. In \autoref{sec:prelim-bc}, we recall the notion of Beck-Chevalley cubes and make these as explicit as possible. In \autoref{sec:prelim-axioms}, we recall the axioms of $ A_n $-schobers.

\subsection{Cubical diagrams}
\label{sec:prelim-cube}
In this section, we recall cubical diagrams in general. We recall the indexing set $ \{0, 1\}^n $ and its interpretation in terms of compositions of $ n+1 $.

\begin{definition}
A \emph{composition} of a number $ n ≥ 0 $ is an ordered tuple of positive natural numbers $ σ = (n_1, …, n_k) $ summing up to $ n $. In case $ n = 0 $ the empty tuple serves as the unique composition of $ n $.
\end{definition}

There are precisely $ 2^n $ compositions of $ n+1 $. They can be enumerated explicitly in terms of their binary presentation which we recall below.

\begin{definition}
The binary presentation map $ ψ_n: \Comp(n+1) → \{0, 1\}^n $ is defined by
\begin{equation*}
ψ_n (n_1, …, n_k) = 0^{n_1 - 1} 1 0^{n_2 - 1} 1 … 0^{n_k - 1}.
\end{equation*}
The symbol $ ψ $ denotes any such $ ψ_n $ generically.
\end{definition}

We have for instance $ ψ(3, 5) = 0010000 $ and $ ψ(3, 5, 5, 3) = 001000010000100 $. We remark for later reference that the palindromic nature of a composition $ σ $ is reflected in its binary presentation $ ψ(σ) $. For two compositions $ σ, τ ∈ \Comp(n+1) $ we write $ σ ≤ τ $ if $ τ $ is a refinement of $ σ $. With respect to the lexical partial ordering on $ \{0, 1\}^n $, the map $ ψ $ is an isomorphism.

\subsection{Bifactorization cubes}
\label{sec:prelim-bifcube}
In this section, we recall the construction of bifactorization cubes. We make the construction as explicit as possible. As indexing set for the vertices we typically use the binary presentation set $ \{0, 1\}^n $.

The definition of bifactorization cubes starts from a pair of two-element compositions $ (a, b), (c, d) ∈ \Comp(n+1) $. We start by introducing a case distinction by relation between the numbers $ a, b, c, d $.

\begin{lemma}
Let $ ((a, b), (c, d)) ∈ \Comp(n+1) $ be a pair of compositions. Then it is of one of the following forms:
\begin{itemize}
\item $ a = c $ cases
\begin{enumerate}
\item $ ((a, b), (c, d)) = ((c, c+m), (c, c+m)) $ with $ c, m > 0 $.
\item $ ((a, b), (c, d)) = ((a, a), (a, a)) $ with $ a > 0 $.
\item $ ((a, b), (c, d)) = ((b+m, b), (b+m, b)) $ with $ b, m > 0 $.
\end{enumerate}
\item $ a > c $ cases
\begin{enumerate}
\item $ ((a, b), (c, d)) = ((c+l, c), (c, c+l)) $ with $ c, l > 0 $.
\item $ ((a, b), (c, d)) = ((b+m+l, b), (b+m, b+l)) $ with $ b, m, l > 0 $.
\item $ ((a, b), (c, d)) = ((c+l, c+m), (c, c+m+l)) $ with $ c, m, l > 0 $.
\end{enumerate}
\item $ a < c $ cases
\begin{enumerate}
\item $ ((a, b), (c, d)) = ((a, a+m), (a+m, a)) $ with $ a, m > 0 $.
\item $ ((a, b), (c, d)) = ((a+l, a+m), (a+m+l, a)) $ with $ a, d, m > 0 $.
\item $ ((a, b), (c, d)) = ((a, a+m+l), (a+m, a+l)) $ with $ a, d, m > 0 $.
\end{enumerate}
\end{itemize}
\end{lemma}

\begin{proof}
Assume $ a = c $. If $ b > a $, we are in the first case. If $ b = a $, we are in the second case. If $ b < a $, we are in the third case. Assume now $ a > c $. If $ b = c $, we are in the first case. If $ b < c $, we are in the second case. If $ b > c $, we are in the third case. Assume now $ a < c $. If $ b = c $, we are in the first case. If $ b < c $, we are in the second case. If $ b > c $, we are in the third case. This exhausts all cases.
\end{proof}

\begin{remark}
In the remainder of \autoref{sec:prelim}, we focus on bifactorization cubes for compositions $ (a, b), (c, d) $ with $ a = c $ or $ a > c $. The description of the bifactorization cubes $ Q_χ ((a, b), (c, d)) $ in case $ a < c $ is rather analogous to the case $ a > c $.
\end{remark}

We are now ready to recall the concept of bifactorization cubes. The input datum is a cube $ χ: \{0, 1\}^n → \Cat_∞ $ and a pair of two-element compositions $ ((a, b), (c, d)) ∈ \Comp(n+1) $. The bifactorization cube $ Q_χ $ is then a cube of categories that are copies of the vertices of $ χ $. More precisely, certain vertices of $ χ $ appear multiple times while others may appear not at all.

\begin{definition}
Let $ χ: \{0, 1\}^n → \Cat_∞ $ be a cubical diagram of categories. Let $ (a, b), (c, d) ∈ \Comp(n+1) $ be a pair of compositions with $ a ≥ c $. Then the \emph{bifactorization cube} $ Q_χ ((a, b), (c, d)) $ is defined as follows:
\begin{enumerate}
\item If $ ((a, b), (c, d)) = ((c, c+m), (c, c+m)) $, then $ Q_χ ((c, c+m), (c, c+m)) $ is $ c+m+1 $-dimensional and we have
\begin{equation*}
Q_χ ((c, c+m), (c, c+m))_{δ_1, δ_2, ε_1, …, ε_{c-1}, ζ, η_1, …, η_{m-1}} = ε_1 … ε_{c-1} (δ_1 ∨ δ_2) ε_{c-1} … ε_1 ζ 0^{m-1}.
\end{equation*}
\item If $ ((a, b), (c, d)) = ((a, a), (a, a)) $, then $ Q_χ ((a, a), (a, a)) $ is $ a+1 $-dimensional and we have
\begin{equation*}
Q_χ ((a, a), (a, a))_{δ_1, δ_2, ε_1, …, ε_{a-1}} = ε_1 … ε_{a-1} (δ_1 ∨ δ_2) ε_{a-1} … ε_1.
\end{equation*}
\item If $ ((a, b), (c, d)) = ((b+m, b), (b+m, b)) $, then $ Q_χ ((b+m, b), (b+m, b)) $ is $ b+2 $-dimensional and we have
\begin{equation*}
Q_χ ((b+m, b), (b+m, b))_{δ_1, δ_2, ζ, ε_1, …, ε_{b-1}} = 0^{m-1} ζ ε_1 … ε_{c-1} (δ_1 ∨ δ_2) ε_{c-1} … ε_1.
\end{equation*}
\item If $ ((a, b), (c, d)) = ((c+l, c), (c, c+l)) $, then $ Q_χ ((c+l, c), (c, c+l)) $ is $ c+1 $-dimensional and
\begin{equation*}
Q((c+l, c), (c, c+l))_{δ_1, δ_2, ε_1, …, ε_{c-1}} = ε_1 … ε_{c-1} (δ_1 0^{l-1} δ_2) ε_{c-1} … ε_1.
\end{equation*}
\item If $ ((a, b), (c, d)) = ((b+m+l, b), (b+m, b+l)) $, then $ Q_χ ((b+m+l, b), (b+m, b+l)) $ is $ b+2 $-dimensional and
\begin{equation*}
Q((b+m+l, b), (b+m, b+l))_{δ_1, δ_2, ζ, ε_1, …, ε_{b-1}} = 0^{m-1} ζ ε_1 … ε_{b-1} δ_1 0^{l-1} δ_2 ε_{c-1} … ε_1.
\end{equation*}
\item If $ ((a, b), (c, d)) = ((c+l, c+m), (c, c+m+l)) $, then $ Q_χ ((c+l, c+m), (c, c+m+l)) $ is $ c+2 $-dimensional and
\begin{equation*}
Q((c+l, c+m), (c, c+m+l))_{δ_1, δ_2, ε_1, …, ε_{c-1}, ζ} = ε_1 … ε_{c-1} δ_1 0^{l-1} δ_2 ε_{c-1} … ε_1 ζ 0^{m-1}.
\end{equation*}
\end{enumerate}
The edge and side fillings of $ Q_χ $ are compositions of the edge and side fillings of $ χ $.
\end{definition}

\begin{remark}
It seems that there are inaccuracies in Definition 3.14 and Example 3.15 of \cite{Dyckerhoff-Wedrich}. Indeed, $ Q (a1, 1a) $ should read $ \{0, 1\} × 0^{a-2} × \{0, 1\} $ instead of the exponent being $ a-1 $, and the indexing of $ Q(ac, ca) $ should be done by the indices of $ Q((l+1)1, 1(l+1)) $ followed by the sequence $ ε_1, …, ε_c $ and not the other way around as the notation potentially suggests. In the present description, we have tried to repair these inaccuracies.
\end{remark}

\subsection{Beck-Chevalley cubes}
\label{sec:prelim-bc}
In this section, we recall the concept of Beck-Chevalley cubes. Extending beyond \cite{Dyckerhoff-Wedrich}, we make the construction as explicit as possible.

The input datum for the construction is a cubical diagram of categories whose edges are exact functors with right adjoints. To avoid confusion with the schober cubical diagram $ χ: \{0, 1\}^n → \Cat_∞ $, we denote this input datum by $ Q: \{0, 1\}^d → \Cat_∞ $. The Beck-Chevalley cube $ \BC(Q) $ is then a cube whose vertices are parallel functors obtained from composing multiple edge maps of $ Q $. The composition of functors can be depicted graphically by stacking the functors on top of each other. We follow the convention of Dyckerhoff-Wedrich:

\begin{convention}[\cite{Dyckerhoff-Wedrich}]
\label{conv:prelim-bc-functorconv}
A sequence of functors $ F_1, …, F_k $ appearing from top to bottom refers to the composition $ F_k ∘ … ∘ F_1 $.
\end{convention}

The functor at an individual vertex of $ \BC(Q) $ is the compositions of four edge functors of $ Q $. In \autoref{def:prelim-bc-def}, we denote this composition by stacking five vertices of $ Q $ upon each other. Every pair of consequent vertices gives rise to one edge functor of $ Q $, and these four functors are composed from top to bottom.

\begin{definition}
\label{def:prelim-bc-def}
Let $ Q: \{0, 1\}^d → \St $ be a cubical diagram of categeries with right adjunctable exact edges. The \emph{Beck-Chevalley cube} $ \BC(Q) $ is the $ (d-1) $-dimensional cube in the functor category $ \Fun(Q_{010^{d-2}}, Q_{100^{d-2}}) $ given as follows.
\begin{itemize}
\item The \emph{top layer} of $ \BC(Q) $ is given by
\begin{equation*}
\BC(Q)_{β, 0} =
\begin{array}[c]{l}
  Q_{0,1,0^{d-2}} \\
  Q_{\mathbf{0,1},β} \\
  Q_{\mathbf{0,0},β} \\
  Q_{\mathbf{1,0},β} \\
  Q_{1,0,0^{d-2}}
\end{array}
\end{equation*}
\item The \emph{bottom layer} of $ \BC(Q) $ is given by
\begin{equation*}
\BC(Q)_{β, 1} =
\begin{array}[c]{l}
  Q_{0,1,0^{d-2}} \\
  Q_{\mathbf{0,1},β} \\
  Q_{\mathbf{1,1},β} \\
  Q_{\mathbf{1,0},β} \\
  Q_{1,0,0^{d-2}}
\end{array}
\end{equation*}
\end{itemize}
\end{definition}

\begin{figure}
\centering
\begin{tikzpicture}[node distance=0]
%
  \node[anchor=north west] (A-1) at (0.25,0) {$ Q_{01,0^{d-2}} $};
  \node[anchor=north west] (A-2) at (A-1.south west) {$ Q_{01,\beta} $};
  \node[anchor=north west] (A-3) at (A-2.south west) {$ Q_{00,\beta} $};
  \node[anchor=north west] (A-4) at (A-3.south west) {$ Q_{10,\beta} $};
  \node[anchor=north west] (A-5) at (A-4.south west) {$ Q_{10,0^{d-2}} $};
%
  \node[anchor=north west] (B-1) at (2.5,0) {$ Q_{01,0^{d-2}} $};
  \node[anchor=north west] (B-2) at (B-1.south west) {$ Q_{01,\beta} $};
  \node[anchor=north west] (B-3) at (B-2.south west) {$ Q_{00,\beta} $};
  \node[anchor=north west] (B-4) at (B-3.south west) {$ Q_{10,\beta} $};
  \node[anchor=north west] (B-5) at (B-4.south west) {$ Q_{10,\beta'} $};
  \node[anchor=north west] (B-6) at (B-5.south west) {$ Q_{10,\beta} $};
  \node[anchor=north west] (B-7) at (B-6.south west) {$ Q_{10,0^{d-2}} $};
%
  \node[anchor=north west] (C-1) at (5,0) {$ Q_{01,0^{d-2}} $};
  \node[anchor=north west] (C-2) at (C-1.south west) {$ Q_{01,\beta} $};
  \node[anchor=north west] (C-3) at (C-2.south west) {$ Q_{00,\beta} $};
  \node[anchor=north west] (C-4) at (C-3.south west) {$ Q_{00,β'} $};
  \node[anchor=north west] (C-5) at (C-4.south west) {$ Q_{10,\beta'} $};
  \node[anchor=north west] (C-6) at (C-5.south west) {$ Q_{10,\beta} $};
  \node[anchor=north west] (C-7) at (C-6.south west) {$ Q_{10,0^{d-2}} $};
%
  \node[anchor=north west] (D-1) at (7.5,0) {$ Q_{01,0^{d-2}} $};
  \node[anchor=north west] (D-2) at (D-1.south west) {$ Q_{01,\beta} $};
  \node[anchor=north west] (D-3) at (D-2.south west) {$ Q_{00,\beta} $};
  \node[anchor=north west] (D-4) at (D-3.south west) {$ Q_{00,β'} $};
  \node[anchor=north west] (D-5) at (D-4.south west) {$ Q_{01,β'} $};
  \node[anchor=north west] (D-6) at (D-5.south west) {$ Q_{00,β'} $};
  \node[anchor=north west] (D-7) at (D-6.south west) {$ Q_{10,\beta'} $};
  \node[anchor=north west] (D-8) at (D-7.south west) {$ Q_{10,\beta} $};
  \node[anchor=north west] (D-9) at (D-8.south west) {$ Q_{10,0^{d-2}} $};
%
%
  \node[anchor=north west] (E-1) at (10,0) {$ Q_{01,0^{d-2}} $};
  \node[anchor=north west] (E-2) at (E-1.south west) {$ Q_{01,\beta} $};
  \node[anchor=north west] (E-3) at (E-2.south west) {$ Q_{00,\beta} $};
  \node[anchor=north west] (E-4) at (E-3.south west) {$ Q_{01,β} $};
  \node[anchor=north west] (E-5) at (E-4.south west) {$ Q_{01,β'} $};
  \node[anchor=north west] (E-6) at (E-5.south west) {$ Q_{00,β'} $};
  \node[anchor=north west] (E-7) at (E-6.south west) {$ Q_{10,\beta'} $};
  \node[anchor=north west] (E-8) at (E-7.south west) {$ Q_{10,\beta} $};
  \node[anchor=north west] (E-9) at (E-8.south west) {$ Q_{10,0^{d-2}} $};
%
  \node[anchor=north west] (F-1) at (12.5,0) {$ Q_{01,0^{d-2}} $};
  \node[anchor=north west] (F-2) at (F-1.south west) {$ Q_{01,β} $};
  \node[anchor=north west] (F-3) at (F-2.south west) {$ Q_{01,β'} $};
  \node[anchor=north west] (F-4) at (F-3.south west) {$ Q_{00,β'} $};
  \node[anchor=north west] (F-5) at (F-4.south west) {$ Q_{10,\beta'} $};
  \node[anchor=north west] (F-6) at (F-5.south west) {$ Q_{10,\beta} $};
  \node[anchor=north west] (F-7) at (F-6.south west) {$ Q_{10,0^{d-2}} $};
%
  \node[anchor=north west] (G-1) at (15,0) {$ Q_{01,0^{d-2}} $};
  \node[anchor=north west] (G-2) at (G-1.south west) {$ Q_{01,β'} $};
  \node[anchor=north west] (G-3) at (G-2.south west) {$ Q_{00,β'} $};
  \node[anchor=north west] (G-4) at (G-3.south west) {$ Q_{10,\beta'} $};
  \node[anchor=north west] (G-5) at (G-4.south west) {$ Q_{10,0^{d-2}} $};
\path[draw, thick, decorate, decoration={mirror, brace}] (A-4.south east) to coordinate[midway] (1-start) (A-4.north east);
\path[draw, thick, decorate, decoration={brace}] (B-6.south west) to coordinate[midway] (1-end) (B-4.north west);
\path[draw, ->] ($ (1-start)!0.2!(1-end) $) -- ($ (1-end)!0.2!(1-start) $) node[midway, above, sloped] {unit};
\path[draw, thick, decorate, decoration={mirror, brace}] (B-5.south east) to coordinate[midway] (2-start) (B-5.south east |- B-3.north east);
\path[draw, thick, decorate, decoration={brace}] (C-5.south west) to coordinate[midway] (2-end) (C-3.north west);
\path[draw, ->] ($ (2-start)!0.2!(2-end) $) -- ($ (2-end)!0.2!(2-start) $) node[midway, above, sloped] {comm};
\path[draw, thick, decorate, decoration={mirror, brace}] (C-4.south east) to coordinate[midway] (3-start) (C-4.north east);
\path[draw, thick, decorate, decoration={brace}] (D-6.south west) to coordinate[midway] (3-end) (D-6.south west |- D-4.north west);
\path[draw, ->] ($ (3-start)!0.2!(3-end) $) -- ($ (3-end)!0.2!(3-start) $) node[midway, above, sloped] {unit};
\path[draw, thick, decorate, decoration={mirror, brace}] (D-5.south east) to coordinate[midway] (4-start) (D-5.south east |- D-3.north east);
\path[draw, thick, decorate, decoration={brace}] (E-5.south west) to coordinate[midway] (4-end) (E-3.north west);
\path[draw, ->] ($ (4-start)!0.2!(4-end) $) -- ($ (4-end)!0.2!(4-start) $) node[midway, above, sloped] {comm};
\path[draw, thick, decorate, decoration={mirror, brace}] (E-4.south east) to coordinate[midway] (5-start) (E-2.north east);
\path[draw, thick, decorate, decoration={brace}] (F-2.south west) to coordinate[midway] (5-end) (F-2.north west);
\path[draw, ->] ($ (5-start)!0.2!(5-end) $) -- ($ (5-end)!0.2!(5-start) $) node[midway, above, sloped] {counit};
\path[draw, thick, decorate, decoration={mirror, brace}] (F-1.north east |- F-3.south east) to coordinate[midway] (6-start) (F-1.north east);
\path[draw, thick, decorate, decoration={mirror, brace}] (F-7.south east) to coordinate[midway] (6B-start) (F-7.south east |- F-5.north east);
\path[draw, thick, decorate, decoration={brace}] (G-2.south west) to coordinate[midway] (6-end) (G-1.north west);
\path[draw, thick, decorate, decoration={brace}] (G-5.south west) to coordinate[midway] (6B-end) (G-4.north west);
\path[draw, ->] ($ (6-start)!0.2!(6-end) $) -- ($ (6-end)!0.2!(6-start) $) node[midway, above, sloped] {func};
\path[draw, ->] ($ (6B-start)!0.2!(6B-end) $) -- ($ (6B-end)!0.2!(6B-start) $) node[midway, above, sloped] {func};
\end{tikzpicture}
\begin{tikzpicture}[scale=0.55]
\newcommand{\bcbasic}{%
\path[draw] (0, 0) coordinate (0-01) -- (2, -0.25) coordinate (0-11) -- (2.5, 0.5) coordinate (0-10) -- (0.5, 0.75) coordinate (0-00) -- cycle;
\begin{scope}[shift={(0, 3)}] \path[draw] (0, 0) coordinate (b-01) -- (2, -0.25) coordinate (b-11) -- (2.5, 0.5) coordinate (b-10) -- (0.5, 0.75) coordinate (b-00) -- cycle; \end{scope}
\begin{scope}[shift={(-1.5, 5.5)}] \path[draw] (0, 0) coordinate (bp-01) -- (2, -0.25) coordinate (bp-11) -- (2.5, 0.5) coordinate (bp-10) -- (0.5, 0.75) coordinate (bp-00) -- cycle; \end{scope}
\path (0-01) ++ (0.35, 0.2) node {\scalebox{0.5}{$ 01 $}};
\path (0-00) ++ (0.1, -0.2) node {\scalebox{0.5}{$ 00 $}};
\path (0-10) ++ (-0.35, -0.2) node {\scalebox{0.5}{$ 10 $}};
\path (0-11) ++ (-0.1, 0.2) node {\scalebox{0.5}{$ 11 $}};
\path (b-01) ++ (0.35, 0.2) node {\scalebox{0.5}{$ 01 $}};
\path (b-00) ++ (0.1, -0.2) node {\scalebox{0.5}{$ 00 $}};
\path (b-10) ++ (-0.35, -0.2) node {\scalebox{0.5}{$ 10 $}};
\path (b-11) ++ (-0.1, 0.2) node {\scalebox{0.5}{$ 11 $}};
\path (bp-01) ++ (0.35, 0.2) node {\scalebox{0.5}{$ 01 $}};
\path (bp-00) ++ (0.1, -0.2) node {\scalebox{0.5}{$ 00 $}};
\path (bp-10) ++ (-0.35, -0.2) node {\scalebox{0.5}{$ 10 $}};
\path (bp-11) ++ (-0.1, 0.2) node {\scalebox{0.5}{$ 11 $}};
\path ($ (0-01)!0.5!(0-10) $) node {\tinymath{$ 0^{d-2} $}};
\path ($ (b-01)!0.5!(b-10) $) node {\tinymath{$ β $}};
\path ($ (bp-01)!0.5!(bp-10) $) node {\tinymath{$ β' $}};
}
\begin{scope}
\bcbasic
\path[draw, ->] ($ (0-01) + (left:0.2) $) to[] ($ (b-01) + (left:0.2) $);
\path[draw, ->] ($ (b-01) + (left:0.2) $) to[] ($ (b-00) + (up:0.2) $);
\path[draw, ->] ($ (b-00) + (up:0.2) $) to[] ($ (b-10) + (0.2, 0.2) $);
\path[draw, ->] ($ (b-10) + (0.2, 0.2) $) to[] ($ (0-10) + (right:0.2) $);
\end{scope}
\begin{scope}[shift={(4, 0)}]
\bcbasic
\path[draw, ->] ($ (0-01) + (left:0.2) $) to[] ($ (b-01) + (left:0.2) $);
\path[draw, ->] ($ (b-01) + (left:0.2) $) to[] ($ (b-00) + (up:0.2) $);
\path[draw, ->] ($ (b-00) + (up:0.2) $) to[] ($ (b-10) + (0.2, 0.2) $);
\path[draw, ->] ($ (b-10) + (0.2, 0.2) $) to[bend right=10] ($ (bp-10) + (right:0.2) $);
\path[draw, ->] ($ (bp-10) + (right:0.2) $) to[bend left=20]  ($ (b-10) + (0.4, 0.4) $);
\path[draw, ->] ($ (b-10) + (0.4, 0.4) $) to[bend left=20] ($ (0-10) + (right:0.2) $);
\end{scope}
\begin{scope}[shift={(8, 0)}]
\bcbasic
\path[draw, ->] ($ (0-01) + (left:0.2) $) to[] ($ (b-01) + (left:0.2) $);
\path[draw, ->] ($ (b-01) + (left:0.2) $) to[] ($ (b-00) + (up:0.2) $);
\path[draw, ->] ($ (b-00) + (up:0.2) $) to[bend left=0] ($ (bp-00) + (up:0.2) $);
\path[draw, ->] ($ (bp-00) + (up:0.2) $) to[bend left=0] ($ (bp-10) + (0.2,0.2) $);
\path[draw, ->] ($ (bp-10) + (0.2,0.2) $) to[bend left=0]  ($ (b-10) + (0.2, 0.2) $);
\path[draw, ->] ($ (b-10) + (0.2, 0.2) $) to[bend left=0] ($ (0-10) + (right:0.2) $);
\end{scope}
\begin{scope}[shift={(12, 0)}]
\bcbasic
\path[draw, ->] ($ (0-01) + (left:0.2) $) to[] ($ (b-01) + (left:0.2) $);
\path[draw, ->] ($ (b-01) + (left:0.2) $) to[] ($ (b-00) + (up:0.2) $);
\path[draw, ->] ($ (b-00) + (up:0.2) $) to[bend left=0] ($ (bp-00) + (up:0.2) $);
\path[draw, ->] ($ (bp-00) + (up:0.2) $) to ($ (bp-01) + (left:0.2) $);
\path[draw, ->] ($ (bp-01) + (left:0.2) $) to[bend left=20] ($ (bp-00) + (up:0.4) $);
\path[draw, ->] ($ (bp-00) + (up:0.4) $) to[bend left=0] ($ (bp-10) + (0.2,0.2) $);
\path[draw, ->] ($ (bp-10) + (0.2,0.2) $) to[bend left=0]  ($ (b-10) + (0.2, 0.2) $);
\path[draw, ->] ($ (b-10) + (0.2, 0.2) $) to[bend left=0] ($ (0-10) + (right:0.2) $);
\end{scope}
\begin{scope}[shift={(16, 0)}]
\bcbasic
\path[draw, ->] ($ (0-01) + (left:0.2) $) to[] ($ (b-01) + (left:0.2) $);
\path[draw, ->] ($ (b-01) + (left:0.2) $) to[] ($ (b-00) + (up:0.2) $);
\path[draw, ->] ($ (b-00) + (up:0.2) $) to[bend right=20] ($ (b-01) + (left:0.4) $);
\path[draw, ->] ($ (b-01) + (left:0.4) $) to ($ (bp-01) + (left:0.2) $);
\path[draw, ->] ($ (bp-01) + (left:0.2) $) to[bend left=20] ($ (bp-00) + (up:0.4) $);
\path[draw, ->] ($ (bp-00) + (up:0.4) $) to[bend left=0] ($ (bp-10) + (0.2,0.2) $);
\path[draw, ->] ($ (bp-10) + (0.2,0.2) $) to[bend left=0]  ($ (b-10) + (0.2, 0.2) $);
\path[draw, ->] ($ (b-10) + (0.2, 0.2) $) to[bend left=0] ($ (0-10) + (right:0.2) $);
\end{scope}
\begin{scope}[shift={(20, 0)}]
\bcbasic
\path[draw, ->] ($ (0-01) + (left:0.2) $) to[] ($ (b-01) + (left:0.2) $);
\path[draw, ->] ($ (b-01) + (left:0.2) $) to ($ (bp-01) + (left:0.2) $);
\path[draw, ->] ($ (bp-01) + (left:0.2) $) to[bend left=20] ($ (bp-00) + (up:0.4) $);
\path[draw, ->] ($ (bp-00) + (up:0.4) $) to[bend left=0] ($ (bp-10) + (0.2,0.2) $);
\path[draw, ->] ($ (bp-10) + (0.2,0.2) $) to[bend left=0]  ($ (b-10) + (0.2, 0.2) $);
\path[draw, ->] ($ (b-10) + (0.2, 0.2) $) to[bend left=0] ($ (0-10) + (right:0.2) $);
\end{scope}
\begin{scope}[shift={(24, 0)}]
\bcbasic
\path[draw, ->] ($ (0-01) + (left:0.2) $) to[] ($ (bp-01) + (left:0.2) $);
\path[draw, ->] ($ (bp-01) + (left:0.2) $) to[bend left=20] ($ (bp-00) + (up:0.4) $);
\path[draw, ->] ($ (bp-00) + (up:0.4) $) to[bend left=0] ($ (bp-10) + (0.2,0.2) $);
\path[draw, ->] ($ (bp-10) + (0.2,0.2) $) to[bend left=0] ($ (0-10) + (right:0.2) $);
\end{scope}
\end{tikzpicture}
\caption{This figure depicts a possible choice of edge maps between vertices in the top layer of the Beck-Chevalley cube $ \BC(Q) $. The upper part of the figure depicts the edge map between the vertices $ \BC(Q)_{β, 0} $ and $ \BC(Q)_{β', 0} $ as the composition of seven natural transformations. The vertex $ \BC(Q)_{β, 0} $ is on the left-most side and the vertex $ \BC(Q)_{β', 0} $ is on the right-most side of the chain of natural transformations. The lower part of the figure depicts the homotopical choices that were made to obtain this chain. The paths on the left-most and right-most graphic are both of length four and stand for the decomposition of four functors. Each of the six steps performs an elementary homotopy operation which translates to a natural transformation of functors. Any other homotopy between the two paths of length four would yield a different choice of edge map between $ \BC(Q)_{β, 0} $ and $ \BC(Q)_{β', 0} $.}
\label{fig:prelim-bc-edgemaptop}
\end{figure}
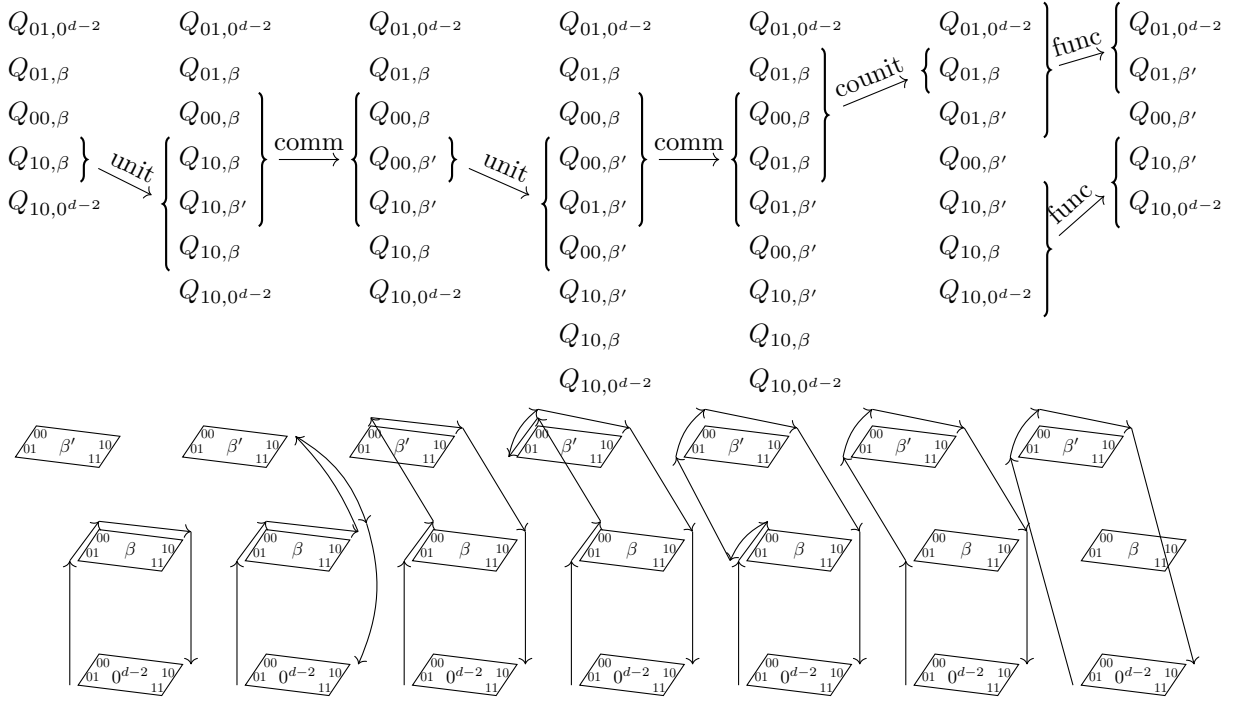

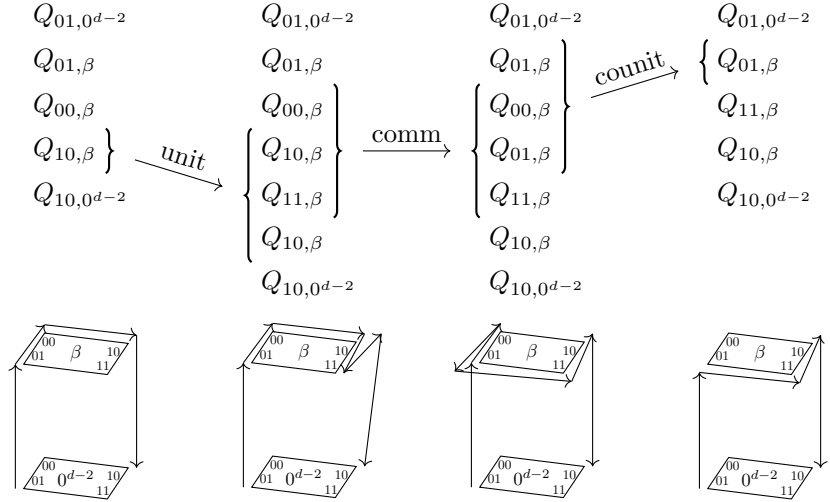
\begin{figure}
\centering
\begin{tikzpicture}[node distance=0]
%
  \node[anchor=north west] (A-1) at (0,0) {$ Q_{01,0^{d-2}} $};
  \node[anchor=north west] (A-2) at (A-1.south west) {$ Q_{01,\beta} $};
  \node[anchor=north west] (A-3) at (A-2.south west) {$ Q_{00,\beta} $};
  \node[anchor=north west] (A-4) at (A-3.south west) {$ Q_{10,\beta} $};
  \node[anchor=north west] (A-5) at (A-4.south west) {$ Q_{10,0^{d-2}} $};
%
  \node[anchor=north west] (B-1) at (3,0) {$ Q_{01,0^{d-2}} $};
  \node[anchor=north west] (B-2) at (B-1.south west) {$ Q_{01,\beta} $};
  \node[anchor=north west] (B-3) at (B-2.south west) {$ Q_{00,\beta} $};
  \node[anchor=north west] (B-4) at (B-3.south west) {$ Q_{10,\beta} $};
  \node[anchor=north west] (B-5) at (B-4.south west) {$ Q_{11,β} $};
  \node[anchor=north west] (B-6) at (B-5.south west) {$ Q_{10,β} $};
  \node[anchor=north west] (B-7) at (B-6.south west) {$ Q_{10,0^{d-2}} $};
%
  \node[anchor=north west] (C-1) at (6,0) {$ Q_{01,0^{d-2}} $};
  \node[anchor=north west] (C-2) at (C-1.south west) {$ Q_{01,\beta} $};
  \node[anchor=north west] (C-3) at (C-2.south west) {$ Q_{00,\beta} $};
  \node[anchor=north west] (C-4) at (C-3.south west) {$ Q_{01,\beta} $};
  \node[anchor=north west] (C-5) at (C-4.south west) {$ Q_{11,β} $};
  \node[anchor=north west] (C-6) at (C-5.south west) {$ Q_{10,β} $};
  \node[anchor=north west] (C-7) at (C-6.south west) {$ Q_{10,0^{d-2}} $};
%
  \node[anchor=north west] (D-1) at (9,0) {$ Q_{01,0^{d-2}} $};
  \node[anchor=north west] (D-2) at (D-1.south west) {$ Q_{01,\beta} $};
  \node[anchor=north west] (D-3) at (D-2.south west) {$ Q_{11,β} $};
  \node[anchor=north west] (D-4) at (D-3.south west) {$ Q_{10,β} $};
  \node[anchor=north west] (D-5) at (D-4.south west) {$ Q_{10,0^{d-2}} $};
\path[draw, thick, decorate, decoration={mirror, brace}] (A-4.south east) to coordinate[midway] (1-start) (A-4.north east);
\path[draw, thick, decorate, decoration={brace}] (B-6.south west) to coordinate[midway] (1-end) (B-4.north west);
\path[draw, ->] ($ (1-start)!0.2!(1-end) $) -- ($ (1-end)!0.2!(1-start) $) node[midway, above, sloped] {unit};
\path[draw, thick, decorate, decoration={mirror, brace}] (B-5.south east) to coordinate[midway] (2-start) (B-5.south east |- B-3.north east);
\path[draw, thick, decorate, decoration={brace}] (C-5.south west) to coordinate[midway] (2-end) (C-3.north west);
\path[draw, ->] ($ (2-start)!0.2!(2-end) $) -- ($ (2-end)!0.2!(2-start) $) node[midway, above, sloped] {comm};
\path[draw, thick, decorate, decoration={mirror, brace}] (C-4.south east) to coordinate[midway] (3-start) (C-2.north east);
\path[draw, thick, decorate, decoration={brace}] (D-2.south west) to coordinate[midway] (3-end) (D-2.south west |- D-2.north west);
\path[draw, ->] ($ (3-start)!0.2!(3-end) $) -- ($ (3-end)!0.2!(3-start) $) node[midway, above, sloped] {counit};
\newcommand{\bcbasictwolayer}{%
\path[draw] (0, 0) coordinate (0-01) -- (2, -0.25) coordinate (0-11) -- (2.5, 0.5) coordinate (0-10) -- (0.5, 0.75) coordinate (0-00) -- cycle;
\begin{scope}[shift={(0, 3)}] \path[draw] (0, 0) coordinate (b-01) -- (2, -0.25) coordinate (b-11) -- (2.5, 0.5) coordinate (b-10) -- (0.5, 0.75) coordinate (b-00) -- cycle; \end{scope}
\path (0-01) ++ (0.35, 0.2) node {\scalebox{0.5}{$ 01 $}};
\path (0-00) ++ (0.1, -0.2) node {\scalebox{0.5}{$ 00 $}};
\path (0-10) ++ (-0.35, -0.2) node {\scalebox{0.5}{$ 10 $}};
\path (0-11) ++ (-0.1, 0.2) node {\scalebox{0.5}{$ 11 $}};
\path (b-01) ++ (0.35, 0.2) node {\scalebox{0.5}{$ 01 $}};
\path (b-00) ++ (0.1, -0.2) node {\scalebox{0.5}{$ 00 $}};
\path (b-10) ++ (-0.35, -0.2) node {\scalebox{0.5}{$ 10 $}};
\path (b-11) ++ (-0.1, 0.2) node {\scalebox{0.5}{$ 11 $}};
\path ($ (0-01)!0.5!(0-10) $) node {\tinymath{$ 0^{d-2} $}};
\path ($ (b-01)!0.5!(b-10) $) node {\tinymath{$ β $}};
}
\begin{scope}[scale=0.55, shift={($ (A-5.south west) + (down:6.5) $)}]
\bcbasictwolayer
\path[draw, ->] ($ (0-01) + (left:0.2) $) to[] ($ (b-01) + (left:0.2) $);
\path[draw, ->] ($ (b-01) + (left:0.2) $) to[] ($ (b-00) + (up:0.2) $);
\path[draw, ->] ($ (b-00) + (up:0.2) $) to[] ($ (b-10) + (0.2, 0.2) $);
\path[draw, ->] ($ (b-10) + (0.2, 0.2) $) to[] ($ (0-10) + (right:0.2) $);
\end{scope}
\begin{scope}[scale=0.55, shift={($ (B-5.south west) + (down:6.5) $)}]
\bcbasictwolayer
\path[draw, ->] ($ (0-01) + (left:0.2) $) to[] ($ (b-01) + (left:0.2) $);
\path[draw, ->] ($ (b-01) + (left:0.2) $) to[] ($ (b-00) + (up:0.2) $);
\path[draw, ->] ($ (b-00) + (up:0.2) $) to[] ($ (b-10) + (0.2, 0.2) $);
\path[draw, ->] ($ (b-10) + (0.2, 0.2) $) to ($ (b-11) + (right:0.2) $);
\path[draw, ->] ($ (b-11) + (right:0.2) $) to ($ (b-10) + (0.6, 0.2) $);
\path[draw, ->] ($ (b-10) + (0.6, 0.2) $) to[] ($ (0-10) + (right:0.2) $);
\end{scope}
\begin{scope}[scale=0.55, shift={($ (C-5.south west) + (down:6.5) $)}]
\bcbasictwolayer
\path[draw, ->] ($ (0-01) + (left:0.2) $) to[] ($ (b-01) + (left:0.2) $);
\path[draw, ->] ($ (b-01) + (left:0.2) $) to[] ($ (b-00) + (up:0.2) $);
\path[draw, ->] ($ (b-00) + (up:0.2) $) to[] ($ (b-01) + (-0.6, -0.2) $);
\path[draw, ->] ($ (b-01) + (-0.6, -0.2) $) to ($ (b-11) + (0.2, -0.2) $);
\path[draw, ->] ($ (b-11) + (0.2, -0.2) $) to ($ (b-10) + (0.2, 0.2) $);
\path[draw, ->] ($ (b-10) + (0.2, 0.2) $) to[] ($ (0-10) + (right:0.2) $);
\end{scope}
\begin{scope}[scale=0.55, shift={($ (D-5.south west) + (down:6.5) $)}]
\bcbasictwolayer
\path[draw, ->] ($ (0-01) + (left:0.2) $) to[] ($ (b-01) + (-0.2, -0.2) $);
\path[draw, ->] ($ (b-01) + (-0.2, -0.2) $) to ($ (b-11) + (0.2, -0.2) $);
\path[draw, ->] ($ (b-11) + (0.2, -0.2) $) to ($ (b-10) + (0.2, 0.2) $);
\path[draw, ->] ($ (b-10) + (0.2, 0.2) $) to[] ($ (0-10) + (right:0.2) $);
\end{scope}
\end{tikzpicture}
\caption{This figure depicts a standard choice of edge map from a vertex in the top layer to the neighbor vertex in the bottom layer of the Beck-Chevalley cube $ \BC(Q) $. The upper part of the figure depicts the edge map $ \BC(Q)_{β, 0} → \BC(Q)_{β, 1} $ and the lower part indicates the underlying homotopical operations, similar to \autoref{fig:prelim-bc-edgemaptop}.}
\label{fig:prelim-bc-edgemapvertical}
\end{figure}

\begin{remark}
Strictly speaking, \autoref{def:prelim-bc-def} is incomplete. For instance, it lacks explanation of the edge and face maps. It is tedious or impossible to give explicit expressions of these maps, because \cite[Construction 3.2]{Dyckerhoff-Wedrich} relies on an existence theorem of Lurie. We however exhibit at least a choice of edge maps explicitly. A choice of edge map between two vertices of the top layer is depicted in \autoref{fig:prelim-bc-edgemaptop}. A choice of edge map between two vertices of the bottom layer can be constructed in a similar fashion. The edge map $ \BC(Q)_{β, 0} → \BC(Q)_{β, 1} $ from top to bottom layer is depicted in \autoref{fig:prelim-bc-edgemapvertical}.
\end{remark}

We finish this section by recalling the notion of higher twist functors.

\begin{definition}
Let $ χ: \Comp(n+1) → \St $ be a cubical diagram with right adjunctable exact edges. Let $ (a, b), (c, d) ∈ \Comp(n+1) $ be a pair of compositions. Then the \emph{twist functor} $ T_{(a, b), (c, d)}: χ_{(a, b)} → χ_{(c, d)} $ is defined as $ T_{(a, b), (c, d)} = \fib(\BC(Q_χ ((a, b), (c, d)))) $.
\end{definition}

\subsection{$ A_n $-schobers}
\label{sec:prelim-axioms}
In this section, we recall the definition of $ A_n $-schobers. Such a schober consists of a cubical diagram of stable $ ∞ $-categories which satisfies certain properties. For the purposes of the present paper, we shall ignore most of the higher $ ∞ $-categorical structure contained in the datum of the cubical diagram, since the NilHecke algebras are ungraded associative algebras and therefore the structure is entirely understood in terms of the abelian world.

\begin{definition}[{\cite[Definition 3.17]{Dyckerhoff-Wedrich}}]
A cubical diagram $ χ: \Comp(n+1) → \St $ of stable $ ∞ $-categories is an \emph{$ A_n $-schober} if the following conditions hold:
\begin{itemize}
\item \textbf{Adjunctability:} All functors in $ χ $ admit right adjoints.
\item \textbf{Recursiveness:} For every proper composition $ (n_1, …, n_k) ∈ \Comp(n+1) $ and every $ 1 ≤ i ≤ k $, the restriction of $ χ $ along $ \Comp(n_i) \embeds \Comp(n+1), τ ↦ (n_1, …, n_{i-1}, τ, n_{i+1}, …, n_k) $ is an $ A_{n_i-1} $-schober.
\item \textbf{Far-commutativity:} For every composition $ (a, b) ∈ \Comp(n+1) $ and compositions $ c_0 ≤ c_1 $ of $ a $ and compositions $ d_0 ≤ d_1 $ of $ b $ the Beck-Chevalley cube of the following square is a natural isomorphism of functors:
\begin{equation*}
\begin{tikzcd}
χ_{(c_0, d_0)} \arrow[r] \arrow[d] & χ_{(c_0, d_1)} \arrow[d] \\
χ_{(c_1, d_0)} \arrow[r] & χ_{(c_1, d_1)}.
\end{tikzcd}
\end{equation*}
\item \textbf{Twist invertibility:} For every composition $ (a, b) ∈ \Comp(n+1) $, the twist functor $ T_{(a, b), (b, a)}: χ_{(a, b)} → χ_{(b, a)} $ is an equivalence.
\item \textbf{Defect vanishing:} For every pair of compositions $ (a, b), (c, d) ∈ \Comp(n+1) $ with $ (a, b) ≠ (d, c) $, the twist functor $ T_{(a, b), (c, d)}: χ_{(a, b)} → χ_{(c, d)} $ is the zero functor.
\end{itemize}
\end{definition}

\begin{remark}
The square in the definition of far-commutativity is a 2-dimensional cubical diagram of categories. Its Beck-Chevalley cube is thus a 1-dimensional diagram of parallel functors, in other words a natural transformation. The far-commutativity axiom demands that this natural transformation is an isomorphism.
\end{remark}


\section{The NilHecke schober}
\label{sec:nh}
In this section we define the NilHecke schober $ \NHS_{n+1} $ and prove that it is an $ A_n $-schober in the sense of Dyckerhoff-Wedrich.

\begin{theorem}[Main theorem]
The cubical diagram $ \NHS_{n+1}: \Comp(n+1) → \St $ is an $ A_n $-schober.
\end{theorem}

In \autoref{sec:nh-algebra}, we start by recalling NilHecke algebras from \cite{ADLSZ} and treat their strand diagram calculus. In \autoref{sec:nh-indres}, we introduce the NilHecke schober and define auxiliary notation and diagrammatic calculus to deal with repeated induction and restriction functors. In \autoref{sec:nh-easy}, we check that $ \NHS_{n+1} $ satisfies the adjunctability, recursiveness and far-commutativity axioms. In \autoref{sec:nh-ccmccm}, we provide a strategy to evaluate Beck-Chevalley cubes for $ \NHS_{n+1} $ and use this strategy to check the defect vanishing axiom for pairs of compositions of the form $ ((c, c+m), (c, c+m)) $. The strategy applies similarly to all other pairs of compositions and is suitable for proving both defect vanishing and twist invertibility axioms. We therefore move considerations for all other types of pairs of compositions to \autoref{sec:appchecks}.

\subsection{NilHecke algebras}
\label{sec:nh-algebra}
In this section, we recall NilHecke algebras. In the visual sense, the NilHecke algebra is a strand algebra with dots and bigon vanishing relation over the base ring $ ℂ⟦ħ⟧ $. In the algebraic sense, the algebra is similar to the smash product of a polynomial algebra with the symmetric group but with modified multiplication operation.

The NilHecke algebra $ \NH_k $ is a $ ℂ⟦ħ⟧ $-linear algebra given as vector space by $ ℂ[X_1, …, X_n] ¤ ℂ[S_n]⟦ħ⟧ $. The strand diagram associated with $ X_1^{k_1} … X_n^{k_n} ¤ s ∈ \NH_k $ is the diagram on $ n $ strands which permutes the strands according to $ s $ and above that applies $ k_1, …, k_n $ dots to the strands. This correspondence is depicted in \autoref{fig:diag-nh-stranddiag}. The multiplication operation of the algebra is different from the semidirect product $ ℂ[X_1, …, X_n] \rtimes ℂ[S_n] $ and incorporates the nil aspect of the algebra. For two basis elements $ f ¤ s $ and $ g ¤ t $, the product $ (f ¤ s) (g ¤ t) $ is given by stacking the corresponding strand diagrams on top of each other and applying a set of rules to arrive at a $ ℂ⟦ħ⟧ $-linear combination of standard strand diagrams. The reduction rules are illustrated in \autoref{fig:diag-nh-diagrules}. We shall fix terminology as follows:

\begin{definition}
Let $ k ≥ 1 $. The \emph{NilHecke algebra} $ \NH_k $ is $ ℂ[X_1, …, X_n] ¤ ℂ[S_n]⟦ħ⟧ $ with the following resolving resolving rules:
\begin{itemize}
\item Bigons vanish in the fashion $ s_i^2 = 0 $.
\item Local dot-pass crossings are resolved in the fashion $ X_i s_i - s_i X_{i+1} = ħ \id $ and $ s_i X_i - X_{i+1} s_i = ħ \id $.
\item Braid crossings are resolved in the form $ s_i s_{i+1} s_i = s_{i+1} s_i s_{i+1} $.
\end{itemize}
For $ τ = (τ_1, …, τ_k) ∈ \Comp(n) $ the \emph{NilHecke algebra} $ \NH_τ $ is defined as
\begin{equation*}
\NH_τ ≔ \NH_{τ_1} ¤ … ¤ \NH_{τ_k}.
\end{equation*}
\end{definition}

\begin{example}
For $ σ = (1, …, 1) ∈ \Comp(n) $ we have $ \NH_σ = ℂ[X_1, …, X_n]⟦ħ⟧ $. For $ σ = (2) ∈ \Comp(2) $ we have $ \NH_σ = ℂ[X_1, X_2] ¤ \vspan(e, s)⟦ħ⟧ $ with $ s X_1 = X_2 s + ħ $ and $ s X_2 = X_1 s - ħ $ and $ s^2 = 0 $.
\end{example}

\begin{convention}
\label{conv:nh-algebra-strandconv}
Strand diagrams are read from bottom to top, and composition $ A · B $ denotes the diagram $ A $ sitting on top of $ B $ as illustrated in \autoref{fig:diag-nh-mult}.
\end{convention}

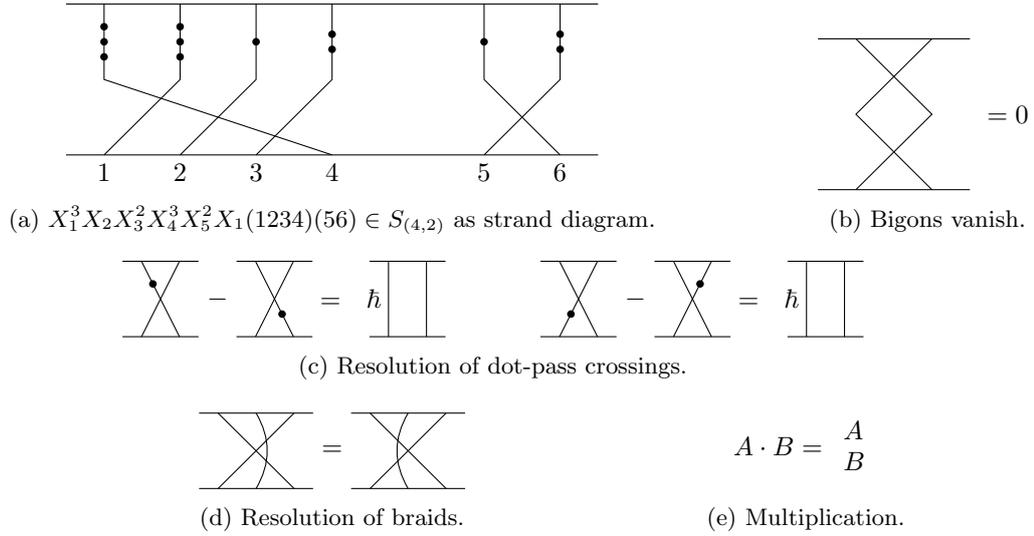
\begin{figure}
\centering
\begin{subfigure}{0.65\linewidth}
\centering
\begin{tikzpicture}
\path[draw] (0.5, 0) -- (7.5, 0);
\path[draw] (0.5, 2) -- (7.5, 2);
\path[draw] (1, 0) node[below] {1} -- (2, 1) -- (2, 2) coordinate[pos=0.3] (A1) coordinate[pos=0.5] (A2) coordinate[pos=0.7] (A3);
\path[draw] (2, 0) node[below] {2} -- (3, 1) -- (3, 2) coordinate[pos=0.5] (B1);
\path[draw] (3, 0) node[below] {3} -- (4, 1) -- (4, 2) coordinate[pos=0.4] (C1) coordinate[pos=0.6] (C2);
\path[draw] (4, 0) node[below] {4} -- (1, 1) -- (1, 2) coordinate[pos=0.3] (D1) coordinate[pos=0.5] (D2) coordinate[pos=0.7] (D3);
\path[draw] (6, 0) node[below] {5} -- (7, 1) -- (7, 2) coordinate[pos=0.4] (E1) coordinate[pos=0.6] (E2);
\path[draw] (7, 0) node[below] {6} -- (6, 1) -- (6, 2) coordinate[pos=0.5] (F1);
\foreach \i in {A1,A2,A3,B1,C1,C2,D1,D2,D3,E1,E2,F1} {\path[fill] (\i) circle[radius=0.05];};
\end{tikzpicture}
\caption{$ X_1^3 X_2 X_3^2 X_4^3 X_5^2 X_1 (1234)(56) ∈ S_{(4, 2)} $ as strand diagram.}
\label{fig:diag-nh-stranddiag}
\end{subfigure}
\hspace{0.05\linewidth}
\begin{subfigure}{0.2\linewidth}
\centering
\begin{tikzpicture}[yscale=1]
\path[draw] (0.5, 0) -- (2.5, 0);
\path[draw] (0.5, 2) -- (2.5, 2);
\path[draw] (1, 0) -- (2, 1) -- (1, 2);
\path[draw] (2, 0) -- (1, 1) -- (2, 2);
\path (3, 1) node {$ = 0 $};
\end{tikzpicture}
\caption{Bigons vanish.}
\label{fig:diag-nh-bigons}
\end{subfigure}
\\[2ex]
\begin{subfigure}{0.64\linewidth}
\centering
\begin{tikzpicture}
\begin{scope}[scale=0.5]
\path[draw] (0.5, 0) -- (2.5, 0);
\path[draw] (0.5, 2) -- (2.5, 2);
\path[draw] (1, 0) -- (2, 2);
\path[draw] (2, 0) -- (1, 2) coordinate[pos=0.7] (A);
\path[fill] (A) circle[radius=0.1];
\path (3, 1) node {$ - $};
\begin{scope}[shift={(3, 0)}]
\path[draw] (0.5, 0) -- (2.5, 0);
\path[draw] (0.5, 2) -- (2.5, 2);
\path[draw] (1, 0) -- (2, 2);
\path[draw] (2, 0) -- (1, 2) coordinate[pos=0.3] (A);
\path[fill] (A) circle[radius=0.1];
\end{scope}
\path (6.5, 1) node {$ = ~~ ħ $};
\begin{scope}[shift={(6.5, 0)}]
\path[draw] (0.5, 0) -- (2.5, 0);
\path[draw] (0.5, 2) -- (2.5, 2);
\path[draw] (1, 0) -- (1, 2);
\path[draw] (2, 0) -- (2, 2);
\end{scope}
\end{scope}
\begin{scope}[scale=0.5, shift={(11, 0)}]
\path[draw] (0.5, 0) -- (2.5, 0);
\path[draw] (0.5, 2) -- (2.5, 2);
\path[draw] (1, 0) -- (2, 2) coordinate[pos=0.3] (A);
\path[draw] (2, 0) -- (1, 2);
\path[fill] (A) circle[radius=0.1];
\path (3, 1) node {$ - $};
\begin{scope}[shift={(3, 0)}]
\path[draw] (0.5, 0) -- (2.5, 0);
\path[draw] (0.5, 2) -- (2.5, 2);
\path[draw] (1, 0) -- (2, 2) coordinate[pos=0.7] (A);
\path[draw] (2, 0) -- (1, 2);
\path[fill] (A) circle[radius=0.1];
\end{scope}
\path (6.5, 1) node {$ = ~~ ħ $};
\begin{scope}[shift={(6.5, 0)}]
\path[draw] (0.5, 0) -- (2.5, 0);
\path[draw] (0.5, 2) -- (2.5, 2);
\path[draw] (1, 0) -- (1, 2);
\path[draw] (2, 0) -- (2, 2);
\end{scope}
\end{scope}
\end{tikzpicture}
\caption{Resolution of dot-pass crossings.}
\label{fig:diag-nh-dotpass}
\end{subfigure}
\\[2ex]

\begin{subfigure}{0.45\linewidth}
\centering
\begin{tikzpicture}[scale=0.5]
\path[draw] (0.5, 0) -- (3.5, 0);
\path[draw] (0.5, 2) -- (3.5, 2);
\path[draw] (1, 0) -- (3, 2);
\path[draw] (3, 0) -- (1, 2);
\path[draw, bend right] (2, 0) to (2, 2);
\begin{scope}[shift={(4, 0)}]
\path[draw] (0.5, 0) -- (3.5, 0);
\path[draw] (0.5, 2) -- (3.5, 2);
\path[draw] (1, 0) -- (3, 2);
\path[draw] (3, 0) -- (1, 2);
\path[draw, bend left] (2, 0) to (2, 2);
\end{scope}
\path (4, 1) node {$ = $};
\end{tikzpicture}
\caption{Resolution of braids.}
\label{fig:diag-nh-nobraid}
\end{subfigure}
\hspace{0.05\linewidth}
\begin{subfigure}{0.2\linewidth}
\centering
\begin{tikzpicture}
\path node {$ A · B = \begin{array}{l} A \\ B \end{array} $};
\end{tikzpicture}
\caption{Multiplication.}
\label{fig:diag-nh-mult}
\end{subfigure}
\caption{This figure illustrates how to interpret elements of the NilHecke algebra as strand diagrams. It also illustrates the relations of the NilHecke algebra in terms of strand diagrams. Multiplication of strand diagrams is given by stacking them upon each other. It is worth remembering that strand diagrams are read from bottom to top. In particular, a crossings in a strand diagrams correspond to the permutation given by reading the strands from bottom to top. For instance, in \autoref{fig:diag-nh-stranddiag} the depicted crossing corresponds to the permutation $ (1234)(56) $.}
\label{fig:diag-nh-diagrules}
\end{figure}

We start by recalling several standard facts relating different NilHecke algebras among each other. When $ σ ∈ \Comp(n+1) $, we define the $ i $-th \emph{block} as the sequence of indices
\begin{equation*}
\sum_{j = 1}^{i-1} σ_j + 1, …, \sum_{j = 1}^{i-1} σ_j + σ_i.
\end{equation*}

\begin{definition}
Let $ σ = (σ_1, …, σ_k) ∈ \Comp(n+1) $ and $ τ = (τ_{1, 1}, …, τ_{1, m_1}, …, τ_{k, 1}, …, τ_{k, m_k}) $ be a refinement of $ σ $. A $ (σ, τ) $-\emph{shuffle} is a permutation $ s ∈ S_{n+1} $ with the following properties:
\begin{itemize}
\item For every $ i = 1, …, k $ and $ j = 1, …, m_j $, the permutation $ s $ maps the $ τ_{i, j} $ block into the $ σ_i $ block.
\item The permutation $ s $ is strictly increasing on each $ τ_{i, j} $ block.
\end{itemize}
\end{definition}

The notion of $ (σ, τ) $-shuffles is illustrated in \autoref{fig:nh-algebra-shuffles}. Note that the finer composition $ τ $ sits below the coarser composition $ σ $. Visually speaking, the smaller boxes shuffle into the larger boxes which sit on top of them. With this definition in mind, the following lemma is an easy graphical consequence.

\begin{lemma}
\label{th:diag-algebra-moduledecomp}
As right $ \NH_τ $-module, we have a natural isomorphism
\begin{equation*}
\NH_σ = \bigoplus_{α ∈ S_{σ, τ}} α \NH_τ.
\end{equation*}
\end{lemma}

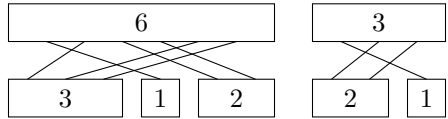
\begin{figure}
\centering
\begin{tikzpicture}
\shufflebox{(0, 0)}{3}{A}{0}
\shufflebox{(1.75, 0)}{1}{B}{0}
\shufflebox{(2.5, 0)}{2}{C}{0}
\shufflebox{(4, 0)}{2}{D}{0}
\shufflebox{(5.25, 0)}{1}{E}{0}
\shufflebox{(0, 1)}{6}{F}{0.25}
\shufflebox{(4, 1)}{3}{G}{0.125}
\path[draw] (A-T1) -- (F-B2);
\path[draw] (A-T2) -- (F-B5);
\path[draw] (A-T3) -- (F-B6);
\path[draw] (B-T1) -- (F-B1);
\path[draw] (C-T1) -- (F-B3);
\path[draw] (C-T2) -- (F-B4);
\path[draw] (D-T1) -- (G-B2);
\path[draw] (D-T2) -- (G-B3);
\path[draw] (E-T1) -- (G-B1);
\path (A-C) node {3};
\path (B-C) node {1};
\path (C-C) node {2};
\path (D-C) node {2};
\path (E-C) node {1};
\path (F-C) node {6};
\path (G-C) node {3};
\end{tikzpicture}
\caption{This figure illustrates the notion of $ (σ, τ) $-shuffles. The characteristic property is that within each $ τ $ box, the outgoing strands connect upwards to strictly increasing slots of the corresponding $ σ $ box. In other words there are no crossings among the strands ascending from each $ τ $ box. In depicted case, we have $ σ = (6, 3) $ and $ τ = (3, 1, 2, 2, 1) $.}
\label{fig:nh-algebra-shuffles}
\end{figure}

\subsection{The NilHecke schober}
\label{sec:nh-indres}
In this section, we introduce the NilHecke schober as a cubical diagram $ \NHS_{n+1}: \{0, 1\}^n → \St $. After providing the definition, we explain the content of the edge maps of the diagram. These adge maps are induction and restriction functors for right modules over NilHecke algebras, and we describe them in terms of diagrammatic calculus.

\begin{definition}
The \emph{NilHecke schober} is the cubical diagram $ \NHS_{n+1}: \Comp(n+1) → \St $ is defined by $ τ ↦ \Perf \NH_τ $ together with the restriction functors $ \Res^σ_τ: \Perf \NH_σ → \Perf \NH_τ $.
\end{definition}

\begin{remark}
We work with right modules over the NilHecke algebras. In particular $ \Perf \NH_σ $ denotes the category of perfect complexes of right modules. Since induction and restriction functors of perfect modules can be developed entirely in the abelian world, we shall be somewhat sloppy with terminology. Given an element $ T ∈ \Perf \NH_σ $, we shall for instance refer to its forgetful object in the category of chain complexes as a vector space instead of a complex. This terminology serves merely to distinguish the $ \NH_σ $-module structure from the underlying linear structure.
\end{remark}

Whenever $ σ ≤ τ $ are two compositions of $ n+1 $, we have an inclusion of algebras $ \NH_τ ⊂ \NH_σ $. This gives rise to a restriction functor $ \Res_τ^σ: \Perf \NH_σ → \Perf \NH_τ $ and its right adjoint $ \Ind_τ^σ: \Perf \NH_τ → \Perf \NH_σ $. They are given by
\begin{equation*}
\Ind_τ^σ (T) = \Hom_{\NH_τ} (\NH_σ, T), \quad \Res_τ^σ (T) = T.
\end{equation*}
The Hom space $ \Hom_{\NH_τ} (\NH_σ, T) $ refers to the right $ \NH_τ $-linear mappings $ \NH_σ → T $. The left $ \NH_σ $-action on $ \NH_σ $ equips $ \Ind_τ^σ (T) $ with a right $ \NH_σ $-action as desired.

\begin{lemma}
\label{th:diag-indres-ind}
Let $ σ ≤ τ $ be two compositions of $ n+1 $ and let $ T ∈ \Perf \NH_τ $. Then as vector spaces we have $ \Ind^σ_τ (T) = \Map(S_{σ, τ}, T) $.
\end{lemma}

\begin{proof}
We have $ \Ind_τ^σ (T) = \Hom_{\NH_τ} (\NH_σ, T) $. By \autoref{th:diag-algebra-moduledecomp}, we have a decomposition of $ \NH_σ $ as right $ \NH_τ $-module into simple $ \NH_τ $-modules each indexed by an element $ s ∈ S_{σ, τ} $. This finishes the proof.
\end{proof}

This means that we can visualize the induction functor $ \Ind_τ^σ $ through shuffle diagrams, which are a subset of the permutation group $ S_{n+1} $. Conversely, we can also visualize the restriction functor $ \Res_τ^σ $ through diagrams. In this interpretation, the restriction functor acts as only one single diagram, namely the identity diagram.

\subsection{Adjunctability, recursiveness and far-commutativity}
\label{sec:nh-easy}
In this section, we show that the cubical diagram $ \NHS_{n+1} $ satisfies the Adjunctability, Recursiveness and Far-commutativity axioms. We proceed by checking all three after each other.

\begin{lemma}
The cubical diagram $ \NHS_{n+1} $ satisfies Adjunctability.
\end{lemma}

\begin{proof}
The induction functor $ \Ind^σ_τ $ is the right adjoint to $ \Res^σ_τ $, which establishes the Adjunctability axiom.
\end{proof}

\begin{lemma}
The cubical diagram $ \NHS_{n+1} $ satisfies Recursiveness.
\end{lemma}

\begin{proof}
Regard a composition $ (n_1, …, n_k) ∈ \Comp(n+1) $. Then the restriction $ \NHS_{n+1} \restr_{\Comp(n_i)} $ along the inclusion map $ \Comp(n_i) → \Comp(n+1) $ given by $ τ ↦ (n_1, …, n_{i-1}, τ, n_{i+1}, … n_k) $ is given by
\begin{equation*}
(\NHS_{n+1} \restr_{\Comp(n_i)})_τ = \Perf(\NH_{n_1} ¤ … ¤ \NH_{n_{i-1}} ¤ \NH_τ ¤ \NH_{n_{i+1}} ¤ … ¤ \NH_{n_k}).
\end{equation*}
The functors in this restricted diagram are given by the forgetful functors between $ \Perf \NH_τ $ and $ \Perf \NH_σ $ for $ σ, τ ∈ \Comp(n_i) $, tensored by identities on the left and right. Checking that this diagram is an $ A_{n_i-1} $-schober comes down to the fact that $ \NHS_{n_i} $ is an $ A_{n_i-1} $-schober, which we shall assume by means of induction.
\end{proof}

\begin{lemma}
The cubical diagram $ \NHS_{n+1} $ satisfies Far-commutativity.
\end{lemma}

\begin{proof}
Regard a composition $ (a, b) ∈ \Comp(n+1) $ and compositions $ c_0 ≤ c_1 $ of $ a $ and $ d_0 ≤ d_1 $ of $ d $. Regard the square
\begin{equation*}
\begin{tikzcd}
\Perf \NH_{c_0, d_0} \arrow[r] \arrow[d] & \Perf \NH_{c_0, d_1} \arrow[d] \\
\Perf \NH_{c_1, d_0} \arrow[r] & \Perf \NH_{c_1, d_1}.
\end{tikzcd}
\end{equation*}
It is our task to show that the BC map of this square is an equivalence. The source and target of the BC map are two functors $ \Perf \NH_{c_0, d_1} → \Perf \NH_{c_1, d_0} $. The first functor is given by first inducing along $ d_0 ≤ d_1 $ and then forgetting along $ c_0 ≤ c_1 $. The second functor is given by first forgetting along $ c_0 ≤ c_1 $ and then inducing along $ d_0 ≤ d_1 $. Since both agree, we finish the proof.
\end{proof}

\subsection{Defect vanishing for $ Q_{\NHS_{2c+m}} ((c, c+m), (c, c+m)) $}
\label{sec:nh-ccmccm}
In this section, we explicitly demonstrate the defect vanishing axiom for one specific type of pair of compositions. We treat the pairs of compositions of the form $ (c, c+m), (c, c+m) $ with $ c, m > 0 $. The total number of strands involved in this pair of compositions is $ n+1 = 2c + m $. The integers $ c, m > 0 $ remain fixed throughout this section. The example $ c = 2 $ and $ m = 1 $ is worked out in \autoref{sec:example-Q_23_23}.

We start by describing the bifactorization cube and Beck-Chevalley cube explicitly. The indices of the Beck-Chevalley cube are binary indices $ \{0, 1\} $. At this moment, we also provide the description of the vertices in binary terms. Later on, we shall describe the vertices more visually in terms of compositions.

\begin{lemma}
\label{th:nh-ccmccm-bifBC}
Let $ c, m > 0 $. Then $ Q_{\NHS_{2c+m}} ((c, c+m), (c, c+m)) $ is $ c+m+1 $-dimensional and we have 
\begin{equation*}
Q_{\NHS_{2c+m}} ((c, c+m), (c, c+m))_{δ_1, δ_2, ε_1, …, ε_{c-1}, ζ, η_1, …, η_{m-1}} = \Perf \NH_{ψ^{-1} (ε_1 … ε_{c-1} (δ_1 ∨ δ_2) ε_{c-1} … ε_1 ζ 0^{m-1})}.
\end{equation*}
The Beck-Chevalley cube $ \BC(Q_{\NHS_{2c+m}} ((c, c+m), (c, c+m)) $ is $ c+m $-dimensional. Its top and bottom layers are given by
\begin{align*}
\BC(Q_{\NHS_{2c+m}} ((c, c+m), (c, c+m)))_{β_1, …, β_{c+m-1}, 0} &=
\begin{array}[c]{l}
\Perf \NH_{ψ^{-1} (0^{c-1} 1 0^{c-1} 0 0^{m-1})} \\ \Perf \NH_{ψ^{-1} (β_1 … β_{c-1} 1 β_{c-1} … β_1 β_c 0^{m-1})} \\ \Perf \NH_{ψ^{-1} (β_1 … β_{c-1} 0 β_{c-1} … β_1 β_c 0^{m-1})} \\ \Perf \NH_{ψ^{-1} (β_1 … β_{c-1} 1 β_{c-1} … β_1 β_c 0^{m-1})} \\ \Perf \NH_{ψ^{-1} (0^{c-1} 1 0^{c-1} 0 0^{m-1})}
\end{array}, \\
\BC(Q_{\NHS_{2c+m}} ((c, c+m), (c, c+m)))_{β_1, …, β_{c+m-1}, 1} &= 
\begin{array}[c]{l}
\Perf \NH_{ψ^{-1} (0^{c-1} 1 0^{c-1} 0 0^{m-1})} \\ \Perf \NH_{ψ^{-1} (β_1 … β_{c-1} 1 β_{c-1} … β_1 β_c 0^{m-1})} \\ \Perf \NH_{ψ^{-1} (β_1 … β_{c-1} 1 β_{c-1} … β_1 β_c 0^{m-1})} \\ \Perf \NH_{ψ^{-1} (β_1 … β_{c-1} 1 β_{c-1} … β_1 β_c 0^{m-1})} \\ \Perf \NH_{ψ^{-1} (0^{c-1} 1 0^{c-1} 0 0^{m-1})}.
\end{array}
\end{align*}
\end{lemma}

\begin{remark}
We have for instance $ ψ^{-1} (0^{c-1} 1 0^{c-1} 0 0^{m-1}) = (c, c+m) $. The difference between the top and bottom layer lies in the third row where the top layer has an inner 0 digit while the bottom layer has an inner 1 digit. The composition $ ψ^{-1} (β_1 … β_{c-1} 0 β_{c-1} … β_1 β_c 0^{m-1}) $ has a block of even width in the middle while in the composition $ ψ^{-1} (β_1 … β_{c-1} 1 β_{c-1} … β_1 β_c 0^{m-1}) $ this block is split into two equal pieces. Note that the left $ 2c $ wide part of all compositions is palindromic.
\end{remark}

Let us now sketch how we compute the total fiber of the Beck-Chevalley cube $ \BC(Q_{\NHS_{2c+m}} ((c, c+m), (c, c+m))) $. We start by observing that the Beck-Chevalley cube is $ c+m $-dimensional. The idea is to compute the total fiber iteratively by taking fibers along different coordinate axes of the cube. We start by setting $ B^{c+m} = \BC(Q_{\NHS_{2c+m}} ((c, c+m), (c, c+m))) $. We then take fibers along the axis from top to bottom layer, and after that we take fibers iteratively with respect to the coordinate axes $ β_{c-1}, …, β_1 $. We use the following notation.

\begin{definition}
The \emph{intermediate Beck-Chevalley cubes} $ B^{c+m-1}, …, B^{c+m-c} $ are given as follows. The cube $ B^{c+m-1} $ is defined by
\begin{equation*}
B^{c+m-1}_{β_1, …, β_{c+m-1}} = \fib(B^{c+m}_{β_1, …, β_{c+m-1}, 0} → B^{c+m}_{β_1, …, β_{c+m-1}, 1}).
\end{equation*}
The cube $ B^{c+m-i} $ is iteratively defined by
\begin{equation*}
B^{c+m-i}_{β_1, …, β_{c-i}, β_c, …, β_{c+m-1}} = \fib(B^{c+m-i+1}_{β_1, …, β_{c-i}, 0, β_c, …, β_{c+m-1}} → B^{c+m-i+1}_{β_1, …, β_{c-i}, 1, β_c, …, β_{c+m-1}}).
\end{equation*}
\end{definition}

We continue this process until we reach a description of $ B^m $. In order to describe the intermediate Beck-Chevalley cubes explicitly, we shall introduce notation to capture compositions which arise during the process.

\begin{definition}
Let $ 0 ≤ i ≤ c-1 $. Let $ β_1, …, β_{c-1-i}, β_c, …, β_{c+m-1} ∈ \{0, 1\} $. Then the \emph{intermediate compositions at level $ i $} are defined as follows:
\begin{align*}
τ &= (τ_1, …, τ_k + i) ≔ ψ^{-1} (β_1, …, β_{c-1-i}, 0^i) ∈ \Comp(c), \\
τ' &= (τ_1, …, τ_k, i) ≔ ψ^{-1} (β_1, …, β_{c-1-i}, 1, 0^{i-1}) ∈ \Comp(c), \\
\tilde{τ} &= (τ_1, …, τ_k + i, i + τ_k, …, τ_1 | m) ≔ \\
& \qquad ψ^{-1} (β_1, …, β_{c-1-i}, 0^i, 0, 0^i, β_{c-1-i}, …, β_1, β_c, 0^{m-1}) ∈ \Comp(2c+m), \\
\tilde{τ'} &= (τ_1, …, τ_k, i, i, τ_k, …, τ_1 | m) ≔ \\
& \qquad ψ^{-1} (β_1, …, β_{c-1-i}, 0^i, 1, 0^i, β_{c-1-i}, …, β_1, β_c, 0^{m-1}) ∈ \Comp(2c+m), \\
\tilde{σ} &= (τ_1, …, τ_k + i + i + τ_k, …, τ_1 | m) ≔ \\
& \qquad ψ^{-1} (β_1, …, β_{c-1-i}, 0^i, 0, 0^i, β_{c-1-i}, …, β_1, β_c, 0^{m-1}) ∈ \Comp(2c+m), \\
\tilde{σ'} &= (τ_1, …, τ_k, i + i, τ_k, …, τ_1 | m) ≔ \\
& \qquad ψ^{-1} (β_1, …, β_{c-1-i}, 0^i, 1, 0^i, β_{c-1-i}, …, β_1, β_c, 0^{m-1}) ∈ \Comp(2c+m).
\end{align*}
In the above formulas, the notation $ τ_1 | m $ denotes the option of $ (τ_1 + m) $ or $ (τ_1, m) $, depending on whether $ β_c = 0 $ or $ β_c = 1 $. In the extremal case $ i = 0 $, the definition of $ τ' $ is simply meant to be $ τ' = (τ_1, …, τ_k) = ψ^{-1} (β_1, …, β_{c-1-i}) $.
\end{definition}

We shall now reformulate the description of the vertices of the Beck-Chevalley cube from \autoref{th:nh-ccmccm-bifBC} directly in terms of compositions. Even though the statement of \autoref{th:nh-ccmccm-BCcubecompositions} makes use of the terminology of intermediate compositions, these in fact refer to level $ 0 $.

\begin{corollary}
\label{th:nh-ccmccm-BCcubecompositions}
Let $ β_1, …, β_{c+m-1} ∈ \{0, 1\} $. Let $ τ, τ', \tilde{τ}, \tilde{τ'}, \tilde{σ}, \tilde{σ'} $ denote the intermediate compositions at level $ 0 $. Then the Beck-Chevalley cube $ B^{c+m} = \BC(Q_{\NHS_{2c+m}} ((c, c+m), (c, c+m))) $ is given as follows:
\begin{align*}
B^{c+m}_{β_1, …, β_{c+m-1}, 0} &=
\begin{array}[c]{l}
\Perf \NH_{c, c+m} \\ \Perf \NH_{\tilde{τ}} \\ \Perf \NH_{\tilde{σ}} \\ \Perf \NH_{\tilde{τ}} \\ \Perf \NH_{c, c+m}
\end{array} \\
B^{c+m}_{β_1, …, β_{c+m-1}, 1} &=
\begin{array}[c]{l}
\Perf \NH_{c, c+m} \\ \Perf \NH_{\tilde{τ'}} \\ \Perf \NH_{\tilde{σ'}} \\ \Perf \NH_{\tilde{τ'}} \\ \Perf \NH_{c, c+m}.
\end{array}
\end{align*}
\end{corollary}

Together with \autoref{th:diag-indres-ind}, we can reformulate the vertices of the Beck-Chevalley cube explicitly. Once an element $ T ∈ \Perf \NH_{c, c+m} $ is given, \autoref{th:nh-ccmccm-BCdiags} describes the vertices explicitly described in terms of mapping spaces from certain sets of shuffle diagrams to $ T $. At this point, it becomes clear that the remainder of the axiom checks is purely combinatorical and the module nature of $ T $ is only used sporadically. We shall therefore drop reference to the module nature as much as possible and focus on explicit descriptions of the sets of shuffle diagrams appearing in the process.

\begin{corollary}
\label{th:nh-ccmccm-BCdiags}
Let $ T ∈ \Perf \NH_{c, c+m} $. Let $ β_1, …, β_{c+m-1} $ and denote by $ \tilde{τ}, \tilde{τ'} $ the associated compositions at level $ 0 $. As vector spaces, we have the following isomorphisms:
\begin{align*}
B^{c+m}_{β_1, …, β_{c+m-1}, 0} &≅ \Map(S_{(c, c+m), \tilde{τ}}, \Map(S_{\tilde{σ}, \tilde{τ}}, T)), \\
B^{c+m}_{β_1, …, β_{c+m-1}, 1} &≅ \Map(S_{(c, c+m), \tilde{τ'}}, \Map(S_{\tilde{σ'}, \tilde{τ'}}, T)).
\end{align*}
\end{corollary}

\begin{lemma}
\label{th:nh-ccmccm-injective}
Let $ β_1, …, β_{c+m-1} ∈ \{0, 1\} $. Then the following composition maps are injective:
\begin{align*}
∘: S_{(c, c+m), \tilde{τ}} × S_{\tilde{σ}, \tilde{τ}} &→ S_{2c+m}, \\
∘: S_{(c, c+m), \tilde{τ'}} × S_{\tilde{σ'}, \tilde{τ'}} &→ S_{2c+m}.
\end{align*}
\end{lemma}

\begin{proof}
Since a shuffle in the first-mentioned set keeps the first $ c $ elements set-wise separated from the second $ c+m $ elements, the shuffle in the second-mentioned set can be reconstructed from the product. The compositions involved in this lemma are illustrated in \autoref{fig:nh-ccmccm-composition}.
\end{proof}

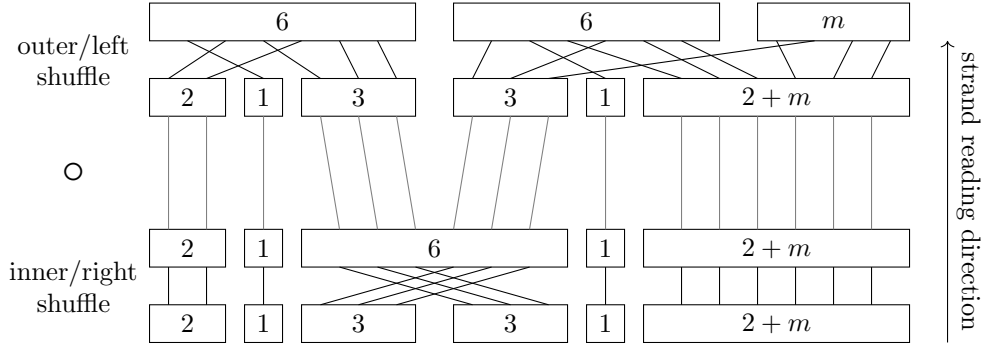
\begin{figure}
\centering
\begin{tikzpicture}
\shufflebox{(0, 0)}{6}{A1}{0.25};
\shufflebox{(4, 0)}{6}{A2}{0.25};
\shufflebox{(8, 0)}{4}{A3}{0};
\path (A1-C) node {6};
\path (A2-C) node {6};
\path (A3-C) node {$ m $};
\shufflebox{(0, -1)}{2}{B1}{0};
\shufflebox{(1.25, -1)}{1}{B2}{0};
\shufflebox{(2, -1)}{3}{B3}{0};
\shufflebox{(4, -1)}{3}{B4}{0};
\shufflebox{(5.75, -1)}{1}{B5}{0};
\shufflebox{(6.5, -1)}{6}{B6}{0.25};
\path (B1-C) node {2};
\path (B2-C) node {1};
\path (B3-C) node {3};
\path (B4-C) node {3};
\path (B5-C) node {1};
\path (B6-C) node {$ 2+m $};
\path[draw] (B1-T1) -- (A1-B2);
\path[draw] (B1-T2) -- (A1-B4);
\path[draw] (B2-T1) -- (A1-B1);
\path[draw] (B3-T1) -- (A1-B3);
\path[draw] (B3-T2) -- (A1-B5);
\path[draw] (B3-T3) -- (A1-B6);
\path[draw] (B4-T1) -- (A2-B1);
\path[draw] (B4-T2) -- (A2-B4);
\path[draw] (B4-T3) -- (A3-B2);
\path[draw] (B5-T1) -- (A2-B2);
\path[draw] (B6-T1) -- (A2-B3);
\path[draw] (B6-T2) -- (A2-B5);
\path[draw] (B6-T3) -- (A2-B6);
\path[draw] (B6-T4) -- (A3-B1);
\path[draw] (B6-T5) -- (A3-B3);
\path[draw] (B6-T6) -- (A3-B4);
\begin{scope}[shift={(0, -3)}]
\shufflebox{(0, 0)}{2}{C1}{0};
\shufflebox{(1.25, 0)}{1}{C2}{0};
\shufflebox{(2, 0)}{6}{C3}{0.25};
\shufflebox{(5.75, 0)}{1}{C5}{0};
\shufflebox{(6.5, 0)}{6}{C6}{0.25};
\path (C1-C) node {2};
\path (C2-C) node {1};
\path (C3-C) node {6};
\path (C5-C) node {1};
\path (C6-C) node {$ 2+m $};
\shufflebox{(0, -1)}{2}{D1}{0};
\shufflebox{(1.25, -1)}{1}{D2}{0};
\shufflebox{(2, -1)}{3}{D3}{0};
\shufflebox{(4, -1)}{3}{D4}{0};
\shufflebox{(5.75, -1)}{1}{D5}{0};
\shufflebox{(6.5, -1)}{6}{D6}{0.25};
\path (D1-C) node {2};
\path (D2-C) node {1};
\path (D3-C) node {3};
\path (D4-C) node {3};
\path (D5-C) node {1};
\path (D6-C) node {$ 2+m $};
\path[draw] (D1-T1) -- (C1-B1);
\path[draw] (D1-T2) -- (C1-B2);
\path[draw] (D2-T1) -- (C2-B1);
\path[draw] (D3-T1) -- (C3-B4);
\path[draw] (D3-T2) -- (C3-B5);
\path[draw] (D3-T3) -- (C3-B6);
\path[draw] (D4-T1) -- (C3-B1);
\path[draw] (D4-T2) -- (C3-B2);
\path[draw] (D4-T3) -- (C3-B3);
\path[draw] (D5-T1) -- (C5-B1);
\path[draw] (D6-T1) -- (C6-B1);
\path[draw] (D6-T2) -- (C6-B2);
\path[draw] (D6-T3) -- (C6-B3);
\path[draw] (D6-T4) -- (C6-B4);
\path[draw] (D6-T5) -- (C6-B5);
\path[draw] (D6-T6) -- (C6-B6);
\end{scope}
\path[draw, gray] (C1-T1) -- (B1-B1);
\path[draw, gray] (C1-T2) -- (B1-B2);
\path[draw, gray] (C2-T1) -- (B2-B1);
\path[draw, gray] (C3-T1) -- (B3-B1);
\path[draw, gray] (C3-T2) -- (B3-B2);
\path[draw, gray] (C3-T3) -- (B3-B3);
\path[draw, gray] (C3-T4) -- (B4-B1);
\path[draw, gray] (C3-T5) -- (B4-B2);
\path[draw, gray] (C3-T6) -- (B4-B3);
\path[draw, gray] (C5-T1) -- (B5-B1);
\foreach \i in {1,...,6} {\path[draw, gray] (C6-T\i) -- (B6-B\i);}
\path (-1, -1.75) node {\LARGE $ \circ $};
\path (-1, -0.25) node[align=center] {outer/left \\ shuffle};
\path (-1, -3.25) node[align=center] {inner/right \\ shuffle};
\path[draw, ->] (10.5, -4) to node[above, sloped, rotate=180] {strand reading direction} (10.5, 0);
\end{tikzpicture}
\caption{This figure illustrates the injectivity of shuffle products appearing in \autoref{th:nh-ccmccm-injective}. Tracing the strands from bottom to top, we see that we can reconstruct both shuffles from the knowledge of their product and the box sizes. Note that the blocks appearing in the diagram have horizontally symmetric lengths due to the nature of the bifactorization cube. The only exception is the extension of the right-most box by length $ m $. The example depicted in the figure concerns the top layer of the Beck-Chevalley cube $ B^{6+m}_{β_1, …, β_{c+m-1}, 0} $ with $ c = 6 $ and arbitrary $ m > 0 $, and $ (β_1, β_2, β_3, β_4, β_5) = (0, 1, 1, 0, 0) $ and $ β_6 = 1 $ and arbitrary $ β_7, …, β_{6+m-1} $.}
\label{fig:nh-ccmccm-composition}
\end{figure}

Along the process we show that the edge maps are split surjections so that the vertices of $ B^{c+m-i} $ are simply kernel subspaces of $ B^{c+m-i+1} $ and ultimately are kernel subspaces of $ B^{c+m} $. The intuition is that in every iteration another zero is inserted into the string of $ β $ indices, so that the inner blocks grow one strand wider in every iteration. To aid this process, we use the following notation.


\begin{definition}
\label{def:nh-ccmccm-anycross}
Let $ 0 ≤ i ≤ c-1 $ and let $ β_1, …, β_{c-1-i}, β_c, …, β_{c+m-1} ∈ \{0, 1\} $. Denote by $ τ, \tilde{τ}, \tilde{σ} $ the intermediate compositions at level $ i $. We define the following subsets:
\begin{itemize}
\item $ \Anycross^i_{τ} ≔ S_{(c, c+m), \tilde{τ}} $.
\item $ \Anycross^i_{τ'} ≔ S_{(c, c+m), \tilde{τ'}} $.
\item $ \Mincross^i_{τ} ⊂ S_{\tilde{σ}, \tilde{τ}} $ is the subset of shuffles which send at least the $ i $ innermost strands of the $ τ_k + i $ block to the right half of the $ τ_k + i + i + τ_k $ block, and at least the $ i $ innermost strands of the $ i + τ_k $ block to the left half of the $ τ_k + i + i + τ_k $ block.
\item $ \Mincross^i_{τ'} ⊂ S_{\tilde{σ'}, \tilde{σ'}} $ is the subset of shuffles which cross over at least the $ i $ innermost strands, similar to $ \Mincross^i_{τ} $.
\end{itemize}
\end{definition}

The notation introduced in \autoref{def:nh-ccmccm-anycross} is slightly imprecise in the sense that the sets on the right-hand side depend on the bit $ β_c $ which is not encoded in $ τ $ or $ τ' $. This is legitimate for our purposes since the notation serves to emphasize the dependence on the intermediate compositions.

\begin{lemma}
Let $ 1 ≤ i ≤ c-1 $. Let $ β_1, …, β_{c-1-i}, β_c, …, β_{c+m-1} ∈ \{0, 1\} $. Let $ τ, τ', \tilde{τ}, \tilde{τ'}, \tilde{σ}, \tilde{σ'} $ denote the intermediate compositions. Then we have the following statements:
\begin{enumerate}
\item $ \Mincross^i_{τ'} = \{F\} $ is a singleton consisting of the permutation which is the total crossing between the two interior $ i $ blocks and the identity on all other blocks.
\item For every $ E ∈ \Anycross^i_{τ'} $, there exists a unique decomposition $ E = E^{(1)} E^{(2)} $ such that $ E^{(1)} ∈ \Anycross^i_{τ} $ and $ E^{(2)} ∈ S_{\tilde{τ}} $ and $ E^{(2)} F ∈ \Mincross^i_{τ} $ and such that the product $ E^{(2)} F $ has no bigons. This map is denoted by
\begin{align*}
δ: \Anycross^i_{τ'} × \Mincross^i_{τ'} &→ \Anycross^i_{τ} × \Mincross^i_{τ}, \\
(E, F) &↦ (E^{(1)}, E^{(2)} F).
\end{align*}
\item The map $ δ $ is injective and its image consists precisely of those pairs $ (S, T) $ where $ T $ has exactly $ i $ internal crossings.
\end{enumerate}
\end{lemma}

\begin{proof}
The first statement is immediate. Regard now the second statement. Enumerate the images of $ E $ on the innermost $ τ_k $, $ i $, $ i $, $ τ_k $ blocks as $ g_1 < … < g_{τ_k + i + i + τ_k} $. Define $ E^{(1)} $ to send the left $ τ_k $ block to the values $ g_1, …, g_{τ_k} $, the left $ i $ block to $ g_{τ_k + 1}, …, g_{τ_k + i} $, the right $ i $ block to $ g_{τ_k + i + 1}, …, g_{τ_k + i + i} $ and the right $ τ_k $ block to $ g_{t_k + i + i + 1}, …, g_{τ_k + i + i + τ_k} $. We have $ E = E^{(1)} E^{(2)} $ where $ E^{(2)} $ is a permutation on the left $ τ_k + i $ block and on the right $ i + τ_k $ block. By definition, we have $ E^{(1)} ∈ \Anycross^i_{τ} $ and $ E^{(2)} ∈ S_{\tilde{τ}} $. Moreover we have $ E^{(1)} F ∈ \Mincross^i_{τ} $ because $ F ∈ \Mincross^i_{τ} $ and $ E^{(1)} $ keeps the left $ τ_k + i $ block and the right $ i + τ_k $ block separated. Similarly, the product $ E^{(1)} F $ has no bigons.

Regard the third statement. From a given image element $ (S, T) = δ(E, F) $, we can reconstruct the product $ EF $ since $ EF = ST $. From $ EF $ we can reconstruct the pair $ (E, F) $, since the product map $ \Anycross^i_{τ'} × \Mincross^i_{τ'} → S_{2c+m} $ is injective. We conclude $ δ $ is injective. Addressing the remaining claims, we note that any $ δ(E, F) $ has precisely $ i $ interior crossings since $ E $ keeps the left $ τ_1, …, τ_k, i $ blocks separate from the right $ i, τ_k, …, τ_1 (+m) $ blocks and $ F $ adds precisely the crossing between the innermost $ i $ blocks. Let now $ (S, T) ∈ \Anycross^i_{τ} × \Mincross^i_{τ} $ be a pair such that $ T $ has the innermost $ i $ strands of the left $ τ_k + i $ block crossing over entirely to the right half of the $ τ_k + i + i + τ_k $ block and the innermost $ i $ strands of the right $ i + τ_k $ block crossing over entirely to the left half of the $ τ_k + i + i + τ_k $ block, but no further overcrossings. Putting $ E = STF $ finishes the proof.
\end{proof}

The next step is to calculate the intermediate Beck-Chevalley cubes $ B^{c+m-i} $ explicitly. We achieve this by iterating from $ i = 0 $ to $ i = c-1 $. The construction proceeds in an object-wise fashion with fixed $ T ∈ \Perf \NH_{c, c+m} $.

\begin{lemma}
\label{th:nh-ccmccm-iteration}
Let $ T ∈ \Perf \NH_{c, c+m} $. We have the following statements for $ i = 0, …, c-1 $.
\begin{enumerate}
\item Let $ β_1, …, β_{c-1-i}, β_c, …, β_{c+m-1} ∈ \{0, 1\} $. Denote by $ τ $ and $ τ' $ the intermediate compositions. Then the following map is a split surjection:
\begin{align*}
π: \Map(\Anycross^i_τ, \Map(\Mincross^i_τ, T)) &→ \Map(\Anycross^i_{τ'}, \Map(\Mincross^i_{τ'}, T)), \\
φ &↦ (E ↦ (F ↦ φ(E^{(1)})(E^{(2)} F))).
\end{align*}
Here $ E = E^{(1)} E^{(2)} $ is the $ δ $ splitting of $ E $, in other words $ (E^{(1)}, E^{(2)} F) = δ(E, F) $.
\item Let $ β_1, …, β_{c-i-1}, β_c, …, β_{c+m-1} ∈ \{0, 1\} $. Denote by $ τ $ and $ τ' $ the intermediate compositions. Then the following diagram commutes:
\begin{equation*}
\begin{tikzcd}
B^{c+m-i}_{β_1, …, β_{c-1-i}, 0, β_c, …, β_{c+m-1)}} (T) \arrow[d, "\sim", "α"'] \arrow[r, "e"] & B^{c+m-i}_{β_1, …, β_{c-i-1}, 1, β_c, …, β_{c+m-1}} (T) \arrow[d, "β"', "\sim"] \\
\Map(\Anycross^i_τ, \Map(\Mincross^i_τ, T)) \arrow[r, "π"] & \Map(\Anycross^i_{τ'}, \Map(\Mincross^i_{τ'}, T)).
\end{tikzcd}
\end{equation*}
Here the upper horizontal map is the edge map of the intermediate Beck-Chevalley cube $ B^{c+m-i} $. In case $ i = 0 $, the indices on the upper horizontal are instead to be read as $ β_1, …, β_{c+m-1}, 0 $ and $ β_1, …, β_{c+m-1}, 1 $.
\item Let $ β_1, …, β_{c-(i+1)}, β_c, …, β_{c+m-1} ∈ \{0, 1\} $. Denote by $ τ $ the intermediate composition at level $ i+1 $. Then there is a natural identification
\begin{equation*}
B^{c+m-(i+1)}_{β_1, …, β_{c-(i+1)}, β_c, …, β_{c+m-1}} (T) ≅ \Map(\Anycross^{i+1}_τ, \Map(\Mincross^{i+1}_τ, T)).
\end{equation*}
\end{enumerate}
\end{lemma}

\begin{proof}
We prove the claim in four parts. In the first part, we show that the first statement holds for every $ i $. In the second part, we show that the second statement holds for $ i = 0 $. In the third part, we show that the second statement implies the third. In the fourth part, we explain that if the second and third statements hold for $ i $, then the second statement holds for $ i+1 $.

For the first part of the proof, we explain that the first statement holds for every $ i $. Indeed, $ δ $ is an injection and an alternative way of writing the map is
\begin{align*}
\tilde{π}: \Map(\Anycross^i_τ × \Mincross^i_τ, T) &→ \Map(\Anycross^i_{τ'} × \Mincross^i_{τ'}, T),  \\
φ &↦ φ ∘ δ.
\end{align*}
%
For the second part of the proof, we explain that the second statement holds in the base case $ i = 0 $. The intermediate compositions are given as follows:
\begin{align*}
τ &= (τ_1, …, τ_k) ≔ ψ^{-1} (β_1, …, β_{c-1}) ∈ \Comp(c), \\
τ' &= τ, \\
\tilde{τ} &= (τ_1, …, τ_k, τ_k, …, τ_1 | m) ≔ ψ^{-1} (β_1, …, β_{c-1}, 0, β_{c-1}, …, β_1, β_c, 0^{m-1}) ∈ \Comp(2c+m), \\
\tilde{τ'} &= (τ_1, …, τ_k, τ_k, …, τ_1 | m) ≔ ψ^{-1} (β_1, …, β_{c-1}, 1, β_{c-1}, …, β_1, β_c, 0^{m-1}) ∈ \Comp(2c+m), \\
\tilde{σ} &= (τ_1, …, τ_k + τ_k, …, τ_1 | m) ≔ ψ^{-1} (β_1, …, β_{c-1}, 0, β_{c-1}, …, β_1, β_c, 0^{m-1}) ∈ \Comp(2c+m), \\
\tilde{σ'} &= (τ_1, …, τ_k, τ_k, …, τ_1 | m) ≔ ψ^{-1} (β_1, …, β_{c-1}, 1, β_{c-1}, …, β_1, β_c, 0^{m-1}) ∈ \Comp(2c+m).
\end{align*}
Regard the edge map, given by restriction and recast:
\begin{equation*}
e: \Hom_{\NH_{τ}} (\NH_{c, c+m}, \Hom_{\NH_τ} (\NH_σ, T)) → \Hom_{\NH_{τ'}} (\NH_{c, c+m}, \Hom_{\NH_{τ'}} (\NH_{σ'}, T)).
\end{equation*}
Denote by $ α $ and $ β $ the identifications as in the diagram. Let $ γ ∈ \Hom_{\NH_{τ}} (\NH_{c, c+m}, \Hom_{\NH_τ} (\NH_σ, T)) $ and $ E ∈ S_{(c, c+m), τ'} $, $ F ∈ S_{(σ', τ')} = \Mincross^0_{τ'} $. Write $ E = E^{(1)} E^{(2)} $ in terms of the $ δ $ splitting. Since $ e(γ) $ is right $ \NH_{τ'} $-linear, we finish the part of the proof by calculating
\begin{align*}
β(e(γ))(E)(F) &= e(γ)(E)(F) = e(γ)(E^{(1)} E^{(2)})(F) \\
& = (e(γ)(E^{(1)}).E^{(2)})(F) = e(γ)(E^{(1)})(E^{(2)} F) = π(α(γ)).
\end{align*}
%
For the third part of the proof, we explain that for any $ i $, the second statement implies the third. Since the second statement holds, it suffices to compute the fiber of $ π $. Since $ π $ is a split surjection, its fiber is precisely the kernel. The kernel of $ π $ consists of those $ φ $ which vanish on the image of $ δ $. Since the image of $ δ $ consists of those pairs $ (S, T) $ where $ T $ has full crossings between the two innnermost $ i $ parts of the $ τ_k + i $ and $ i + τ_k $ blocks but no more, the statement $ φ ∈ \Ker(π) $ is equivalent to $ φ $ taking only values on those pairs that swap at least $ i+1 $ of the innermost strands. This proves the third statement.

The fourth part of the proof is analogous to the second part of the proof, in the sense that once the intermediate Beck-Chevalley cube $ B^{c+m-(i+1)} $ is established, the new edge maps are simply restrictions of the existing edge maps to the kernels. This implies the second statement for $ i+1 $ because the value on a pair $ (E, F) $ under the restriction of an element $ φ $ to the kernel of $ π $ is determined by the value of $ φ $ on an appropriate alternative splitting of the product $ EF $. This finishes the proof.
\end{proof}

We are now ready to inspect the total fiber of the Beck-Chevalley cube.

\begin{proposition}
\label{th:nh-ccmccm-vanishing}
The total fiber of $ \BC(Q_{\NHS_{2c+m}} ((c, c+m), (c, c+m))) $ vanishes.
\end{proposition}

\begin{proof}
Let $ β_c, …, β_{c+m-1} ∈ \{0, 1\} $. Using $ i = c-1 $ in \autoref{th:nh-ccmccm-iteration}, we obtain
\begin{equation*}
B^m_{β_c, …, β_{c+m-1}} (T) ≅ \Map(\Anycross^c_τ, \Map(\Mincross^c_τ, T)).
\end{equation*}
Let us describe the sets $ \Anycross^c_τ $ and $ \Mincross^c_τ $ explicitly, depending on the value of the remaining indices $ β_c, …, β_{c+m-1} $:
\begin{itemize}
\item In case $ β_c = 0 $, the intermediate compositions associated with $ β_c, …, β_{c+m-1} $ at level $ c-1 $ are $ \tilde{τ} = (c, c + m) $ and $ \tilde{σ} = (c+c + m) $. We have $ \Anycross^c_τ = \{\id\} $. An element $ F ∈ \Mincross^c_τ $ amounts to a permutation which, with respect to the source composition $ (c, c, m) $ and the target composition $ (c, c+m) $, shuffles the first source $ c $ block into the target $ c+m $ block and the second source $ c $ block into the target $ c $ block and the source $ m $ block into the target $ c + m $ block.
\item In case $ β_c = 1 $, the intermediate compositions associated with $ β_c, …, β_{c+m-1} $ at level $ c-1 $ are $ \tilde{τ} = (c, c, m) $ and $ \tilde{σ} = (c+c, m) $. We have $ \Anycross^c_τ = \id_c × S_{(c+m), (c, m)} $ and $ \Mincross^c_τ = \{X_c × \id_m\} $ where $ X_c ∈ S_{2c} $ denotes the full crossing of the first $ c $ and second $ c $ blocks.
\end{itemize}
In both cases the set $ \Anycross^c_τ × \Mincross^c_τ $ is identified with $ S_{(c+m), (c, m)} $ and we conclude that $ B^{m-1} $ is the zero cube. This finishes the proof.
\end{proof}

\appendix
\section{Examples}
\label{sec:examples}
In this section we demonstrate our overall evaluation strategy for the total fibers of Beck-Chevalley cubes in examples. The first example is the $ A_2 $-schober $ \NHS_{3} $. The second example is the total fiber of the specific Beck-Chevalley cube $ \BC(Q_{\NHS_5} ((2, 3), (2, 3))) $, which is evaluated along the procedure of \autoref{sec:nh-ccmccm}.

\subsection{The $ \NH_3 $ schober}
In this section, we illustrate the $ A_2 $-schober axioms of the $ \NH_3 $-schober. We freely use the notation $ I, F, G, H $ as in \cite[section 2]{Dyckerhoff-Wedrich}. We address all $ A_2 $-schober axioms and explain in elementary manners why they hold. The schober diagram $ χ: \{0, 1\}^2 → \Cat_∞ $ is given as follows:
\begin{equation*}
\begin{tikzcd}
\Perf \NH_3 \arrow[r, "I"] \arrow[d, "H"] & \Perf \NH_{1,2} \arrow[d, "F"] \\
\Perf \NH_{2,1} \arrow[r, "G"] & \Perf \NH_{1, 1, 1}
\end{tikzcd}
\end{equation*}

Note that in this diagram the first axis is the horizontal axis, oriented from left to right, and the second axis is the vertical axis, oriented from top to bottom.

\paragraph*{Preliminaries.}
In this section, we set up basic notation for working out the $ A_2 $-schober axioms for the $ \NHS_3 $ schober. It is clear that adjunctability, recursiveness and far-commutativity hold. We shall therefore focus on the bifactorization cubes and their Beck-Chevalley cubes. We shall focus on the bifactorization cubes $ Q_{\NHS_3} ((1, 2), (1, 2)) $ and $ Q_{\NHS_3} ((1, 2), (2, 1)) $ since the other two bifactorization cubes $ Q_{\NHS_3} ((2, 1), (1, 2)) $ and $ Q_{\NHS_3} ((2, 1), (2, 1)) $ are dealt with in a similar fashion.

We start by writing $ \NH_3 = III \NH_{1, 2} ⊕ XI \NH_{1, 2} ⊕ W \NH_{1, 2} $ where we use the symbols

\begin{center}
\begin{tikzpicture}
\begin{scope}
\path[draw] (0, 0) -- (1, 0) coordinate[pos=0.25] (A-bot) coordinate[pos=0.5] (B-bot) coordinate[pos=0.75] (C-bot);
\path[draw] (0, 1) -- (1, 1) coordinate[pos=0.25] (A-top) coordinate[pos=0.5] (B-top) coordinate[pos=0.75] (C-top);
\path[draw] (A-top) -- (A-bot);
\path[draw] (B-top) -- (B-bot);
\path[draw] (C-top) -- (C-bot);
\path (0.5, -0.5) node {$ III $};
\end{scope}
\begin{scope}[shift={(3, 0)}]
\path[draw] (0, 0) -- (1, 0) coordinate[pos=0.25] (A-bot) coordinate[pos=0.5] (B-bot) coordinate[pos=0.75] (C-bot);
\path[draw] (0, 1) -- (1, 1) coordinate[pos=0.25] (A-top) coordinate[pos=0.5] (B-top) coordinate[pos=0.75] (C-top);
\path[draw] (A-top) -- (B-bot);
\path[draw] (B-top) -- (A-bot);
\path[draw] (C-top) -- (C-bot);
\path (0.5, -0.5) node {$ XI $};
\end{scope}
\begin{scope}[shift={(6, 0)}]
\path[draw] (0, 0) -- (1, 0) coordinate[pos=0.25] (A-bot) coordinate[pos=0.5] (B-bot) coordinate[pos=0.75] (C-bot);
\path[draw] (0, 1) -- (1, 1) coordinate[pos=0.25] (A-top) coordinate[pos=0.5] (B-top) coordinate[pos=0.75] (C-top);
\path[draw] (A-top) -- (B-bot);
\path[draw] (B-top) -- (C-bot);
\path[draw] (C-top) -- (A-bot);
\path (0.5, -0.5) node {$ W $};
\end{scope}
\end{tikzpicture}
\end{center}

By means of shorthand notation, we shall denote all inductions simply by $ \Hom(\NH_{1, 2}, -) $ and do not write out forgetful functors. Whenever an inductions to $ \NH_{1, 2} $ and $ \NH_{2, 1} $ is done then the module is implicitly casted to a $ \NH_{1, 1, 1} $-module before, and whenever an induction to $ \NH_3 $ is done then the module is supposed to be either a $ \NH_{1, 2} $- or $ \NH_{2, 1} $-module. In the single case of $ I H^* H I^* $ there arises notational ambiguity, which we therefore make explicit.

\paragraph*{The bifactorization cube $ Q_{\NHS_3} ((1, 2), (2, 1)) $.}
\label{sec:example-1221}
In this section, we compute the bifactorization cube $ Q_{\NHS_3} ((1, 2), (2, 1)) $ and its Beck-Chevalley cube. The bifactorization cube is in fact 2-dimensional so that its Beck-Chevalley cube is 1-dimensional.

This specific bifactorization cube is not described in \autoref{sec:prelim-bifcube}, but according to the procedure of Dyckerhoff-Wedrich it is given by
\begin{equation*}
\begin{tikzcd}
\Perf \NH_3 \arrow[r] \arrow[d] & \Perf \NH_{2, 1} \arrow[d] \\
\Perf \NH_{1,2} \arrow[r] & \Perf \NH_{1,1,1}.
\end{tikzcd}
\end{equation*}
As usual, the $ (1,0) $-indexed entry is the top-right entry, and the $ (0,1) $-indexed entry is the bottom-left entry. Let us now compute the Beck-Chevalley cube, which is in fact a 1-dimensional diagram. Its top layer has the following single entry:
\begin{align*}
\BC(Q_{\NHS_3} ((1, 2), (2, 1)))_{0} &= \begin{array}[c]{l}
  Q_{\NHS_3} ((1, 2), (2, 1))_{01} \\
  Q_{\NHS_3} ((1, 2), (2, 1))_{01} \\
  Q_{\NHS_3} ((1, 2), (2, 1))_{00} \\
  Q_{\NHS_3} ((1, 2), (2, 1))_{10} \\
  Q_{\NHS_3} ((1, 2), (2, 1))_{10}
\end{array}
=
\begin{array}[c]{l}
  \Perf \NH_{1,2} \\
  \Perf \NH_{1,2} \\
  \Perf \NH_{3} \\
  \Perf \NH_{2,1} \\
  \Perf \NH_{2,1}
\end{array}
= H I^*.
\end{align*}
Its bottom layer has the following single entry:
\begin{align*}
\BC(Q_{\NHS_3} ((1, 2), (2, 1)))_{1} &= \begin{array}[c]{l}
  Q_{\NHS_3} ((1, 2), (2, 1))_{01} \\
  Q_{\NHS_3} ((1, 2), (2, 1))_{01} \\
  Q_{\NHS_3} ((1, 2), (2, 1))_{11} \\
  Q_{\NHS_3} ((1, 2), (2, 1))_{10} \\
  Q_{\NHS_3} ((1, 2), (2, 1))_{10}
\end{array}
=
\begin{array}[c]{l}
  \Perf \NH_{1,2} \\
  \Perf \NH_{1,2} \\
  \Perf \NH_{1,1,1} \\
  \Perf \NH_{2,1} \\
  \Perf \NH_{2,1}
\end{array}
= G^* F.
\end{align*}
The Beck-Chevalley cube is thus the single natural transformation
\begin{equation*}
\BC(Q_{\NHS_3} ((1, 2), (2, 1)))  = (HI^* ⇒ G^*F).
\end{equation*}

\begin{lemma}
Let $ T ∈ \Perf \NH_{1, 2} $. The map $ (HI^*)(T) → (G^* F)(T) $ is the natural restriction and recast $ \Hom_{\NH_{1,2}} (\NH_3, T) → \Hom_{\NH_{1,1,1}} (\NH_{2,1}, T) $.
\end{lemma}

\begin{proof}
We shall compute this map explicitly. The $ \NH_{2, 1} $-modules $ H I^* T $ and $ G^* F $ are of the form $ H I^* T = \Hom(\NH_3, T) $ and $ G^* F T = \Hom(\NH_{2, 1}, T) $.

The map $ T → F^* F T $ is given by $ t ↦ (n ↦ tn) $. Applying $ HI^* $ on both sides yields the map
\begin{align*}
\Hom(\NH_3, T) &→ \Hom(\NH_3, F^* F T), \\
φ \quad &↦ \quad [(t ↦ (n ↦ tn)) ∘ φ] = [m ↦ (n ↦ φ(m) n)].
\end{align*}
Next we have the map
\begin{align*}
\Hom(\NH_3, F^* F T) &→ \Hom(\NH_3, F T), \\
ψ \quad &↦ \quad n ↦ ψ(n)(1).
\end{align*}
We have the map
\begin{align*}
\Hom(\NH_3, F T) &→ \Hom(\NH_3, G^* F T), \\
ψ \quad &↦ \quad n ↦ (m ↦ ψ(nm)).
\end{align*}
The map $ \Hom(\NH_3, G^* F T) → G^* F T $ is simply the evaluation map at $ 1 ∈ \NH_3 $.

We now regard the composition
\begin{align*}
H \Hom(\NH_3, T) &→ H \Hom(\NH_3, \Hom(\NH_{1, 2}, FT)) \\
φ \quad &↦ \quad m ↦ (n ↦ φ(m) n) \\
&→ \Hom(\NH_3, F T) \\
&↦ φ \\
& \isoto H \Hom(\NH_3, \Hom(\NH_{2, 1}, FT)) \\
&↦ \quad n ↦ (m ↦ φ(nm)) \\
& → \Hom(\NH_{2, 1}, FT) \\
&↦ \quad φ.
\end{align*}
Simply speaking, $ φ $ starts its life as a $ \NH_{1, 2} $-module map $ \NH_3 → T $. Throughout the process it gets restricted and becomes a $ \NH_{1, 1, 1} $-module map $ \NH_{2, 1} → T $. This finishes the proof.
\end{proof}

\begin{lemma}
The kernel of the Beck-Chevalley map is the natural isomorphism $ \Perf \NH_{1,2} → \Perf \NH_{2,1} $ induced from the flip isomorphism $ \NH_{1, 2} ≅ \NH_{2, 1} $.
\end{lemma}

\begin{proof}
A morphism $ φ ∈ \Hom(\NH_3, T) $ is determined by $ φ(III), φ(XI), φ(W) ∈ T $. In our case, the map is a split surjection, and the fiber is simply the kernel. In order to lie in the kernel, we need precisely that $ φ(III) = φ(XI) = 0 $.

We shall now prove that the kernel is isomorphic to $ T $, where $ T $ is viewed as $ \NH_{2, 1} $-module via the algebra isomorphism $ ψ: \NH_{2, 1} → \NH_{1, 2} $ given by $ ψ(XI) = IX $, and $ ψ(ijk) = kij $ on dots-ony strand diagrams. Note that $ ψ $ is different from the mirroring map, which instead sends $ ijk $ to $ kji $. Finally, we remark that $ ψ $ has the property that $ N · W = W · ψ(N) $ for $ N ∈ \NH_{2, 1} $.

Now let us investigate the $ \NH_{2, 1} $-action on the kernel space. The element $ XI ∈ \NH_{2, 1} $ acts on the function $ φ(W) = t $ as $ (φ.XI)(III) = φ(XI · III) = 0 $ and $ (φ.XI)(XI) = φ(XI · XI) = 0 $ and $ (φ.XI)(W) = φ(XI · W) = φ(W · IX) = φ(W) · IX $. More generally for $ N ∈ \NH_{2, 1} $ we have $ (φ.N)(W) = φ(N·W) = φ(W · ψ(N)) = φ(W) · ψ(N) $. This way, we observe that the kernel is isomorphic to $ T $ with $ \NH_{2, 1} $-action on $ T $ given by $ t.n = t.ψ(n) $. In other words, the kernel is precisely $ T $ interpreted as right $ \NH_{2, 1} $-module via the standard flip. This finishes the proof.
\end{proof}

\paragraph*{The bifactorization cube $ Q_{\NHS_3} ((1, 2), (1, 2)) $.}
In this section, we regard the bifactorization cube $ Q_{\NHS_3} ((1, 2), (1, 2)) $ and its Beck-Chevalley cube. The bifactorization cube is 3-dimensional and its Beck-Chevalley cube is 2-dimensional.

The bifactorization cube is given by
\begin{equation*}
\begin{tikzcd}
  \Perf \NH_3 \arrow[rr] \arrow[dd] \arrow[dr] && \Perf \NH_{1,2} \arrow[dd] \arrow[dr] & \\
  & \Perf \NH_{2,1} \arrow[rr] \arrow[dd] && \Perf \NH_{1,1,1} \arrow[dd] \\
  \Perf \NH_{1,2} \arrow[rr] \arrow[dr] && \Perf \NH_{1,2} \arrow[dr] & \\
  & \Perf \NH_{1,1,1} \arrow[rr] && \Perf \NH_{1,1,1}
\end{tikzcd}
\end{equation*}
In this notation, the rear face captures the entries $ Q_{\NHS_3} ((1, 2), (1, 2))_{*, *, 0} $ and the front face captures the entries $ Q_{\NHS_3} ((1, 2), (1, 2))_{*, *, 1} $. Let us now compute the Beck-Chevalley cube. Its top layer has the following entries:
\begin{align*}
\BC(Q_{\NHS_3} ((1, 2), (1, 2)))_{00} &= \begin{array}[c]{l}
  Q_{\NHS_3} ((1, 2), (1, 2))_{01,0} \\
  Q_{\NHS_3} ((1, 2), (1, 2))_{01,0} \\
  Q_{\NHS_3} ((1, 2), (1, 2))_{00,0} \\
  Q_{\NHS_3} ((1, 2), (1, 2))_{10,0} \\
  Q_{\NHS_3} ((1, 2), (1, 2))_{10,0}
\end{array}
=
\begin{array}[c]{l} \Perf \NH_{1, 2} \\ \Perf \NH_{1, 2} \\ \Perf \NH_3 \\ \Perf \NH_{1, 2} \\ \Perf \NH_{1, 2} 
\end{array}
= II^*, \\
\BC(Q_{\NHS_3} ((1, 2), (1, 2)))_{10} &= \begin{array}[c]{l}
  Q_{\NHS_3} ((1, 2), (1, 2))_{01,0} \\
  Q_{\NHS_3} ((1, 2), (1, 2))_{01,1} \\
  Q_{\NHS_3} ((1, 2), (1, 2))_{00,1} \\
  Q_{\NHS_3} ((1, 2), (1, 2))_{10,1} \\
  Q_{\NHS_3} ((1, 2), (1, 2))_{10,0}
\end{array}
=
\begin{array}[c]{l} \Perf \NH_{1, 2} \\ \Perf \NH_{1,1,1} \\ \Perf \NH_{2,1} \\ \Perf \NH_{1,1,1} \\ \Perf \NH_{1, 2}
\end{array}
= F^* G G^* F.
\end{align*}
Its bottom layer has the following entries:
\begin{align*}
\BC(Q_{\NHS_3} ((1, 2), (1, 2)))_{01} &= \begin{array}[c]{l}
  Q_{\NHS_3} ((1, 2), (1, 2))_{01,0} \\
  Q_{\NHS_3} ((1, 2), (1, 2))_{01,0} \\
  Q_{\NHS_3} ((1, 2), (1, 2))_{11,0} \\
  Q_{\NHS_3} ((1, 2), (1, 2))_{10,0} \\
  Q_{\NHS_3} ((1, 2), (1, 2))_{10,0}
\end{array}
=
\begin{array}[c]{l} \Perf \NH_{1, 2} \\ \Perf \NH_{1, 2} \\ \Perf \NH_{1,2} \\ \Perf \NH_{1, 2} \\ \Perf \NH_{1, 2}
\end{array}
= \Id_{\Perf \NH_{1,2}}, \\
\BC(Q_{\NHS_3} ((1, 2), (1, 2)))_{11} &= \begin{array}[c]{l}
  Q_{\NHS_3} ((1, 2), (1, 2))_{01,0} \\
  Q_{\NHS_3} ((1, 2), (1, 2))_{01,1} \\
  Q_{\NHS_3} ((1, 2), (1, 2))_{11,1} \\
  Q_{\NHS_3} ((1, 2), (1, 2))_{10,1} \\
  Q_{\NHS_3} ((1, 2), (1, 2))_{10,0}
\end{array}
=
\begin{array}[c]{l} \Perf \NH_{1, 2} \\ \Perf \NH_{1,1,1} \\ \Perf \NH_{1,1,1} \\ \Perf \NH_{1,1,1} \\ \Perf \NH_{1, 2}
\end{array}
= F^* F.
\end{align*}
In total, the Beck-Chevalley cube $ \BC(Q_{\NHS_3} ((1, 2), (1, 2))) $ reads
\begin{equation*}
\begin{tikzcd}
II^* \arrow[d] \arrow[r] & F^* G G^* F \arrow[d] \\
\Id_{\Perf \NH_{1, 2}} \arrow[r] \arrow[r] & F^* F
\end{tikzcd}
\end{equation*}

\begin{lemma}
The Beck-Chevalley square reads explicitly
\begin{equation*}
\begin{tikzcd}
\Hom(\NH_3, T) \arrow[r] \arrow[d] & \Hom(\NH_{1, 2}, \Hom(\NH_{2, 1}, T)) \arrow[d] \\
T \arrow[r] & \Hom(\NH_{1, 2}, T).
\end{tikzcd}
\end{equation*}
\end{lemma}

\begin{proof}
We have the composition functors
\begin{align*}
I H^* ∘ H I^*: \Perf \NH_{1, 2} &→ \Perf \NH_{2, 1} → \Perf \NH_{1, 2}, \\
F^* G ∘ G F^*: \Perf \NH_{1, 2} &→ \Perf \NH_{2, 1} → \Perf \NH_{1, 2}.
\end{align*}
It is our task to compute the chained natural morphism $ (I H^* ∘ H I^*) T → (F^* G ∘ G^* F) T $  for every $ T ∈ \Perf \NH_{1, 2} $. Both domain and codomain are in $ \Perf \NH_{1, 2} $.

First we construct the morphism $ H I^* T → G^* F T $. It is the following, simply the restriction:
\begin{equation*}
\Hom(\NH_3, T) → \Hom(\NH_{2, 1}, T).
\end{equation*}
Second we apply $ I H^* $ to both sides. The resulting map acts simply as forgetful recast in the inner hom:
\begin{equation*}
\Hom(\NH_3, \Hom(\NH_3, T)) → \Hom(\NH_3, \Hom(\NH_{2, 1}, T)).
\end{equation*}
Third we compute the map $ I H^* (G^* F T) → F^* G (G^* F T) $. This is given by the forgetful recast in the outer hom:
\begin{equation*}
\Hom(\NH_3, \Hom(\NH_{2, 1}, T)) → \Hom(\NH_{1, 2}, \Hom(\NH_{2, 1}, T))
\end{equation*}
Fourth we compute the map $ I I^* T → I H^* H I^* T $. This is simply the standard extension map. Note that the inner $ \Hom(\NH_3, T) $ is preserved and the outer $ \Hom $ is new:
\begin{align*}
\Hom(\NH_3, T) &→ \Hom(\NH_3, H \Hom(\NH_3, T)), \\
φ \quad &↦ \quad n ↦ (m ↦ φ(nm)).
\end{align*}
The overall composition $ I I^* T → (F^* G ∘ G^* F) T $ is the composition of these three maps and thus
sends $ φ $ to its standard extension and then inner and outer recasts.

Furthermore the object $ F^* F T $ is $ \Hom(\NH_{1, 2}, T) $ and the map $ F^* G G^* F T → F^* F T $ is given by evaluating the inner hom at $ 1 $. This finishes the proof.
\end{proof}

Proving defect-vanishing for the pair of compositions $ ((a, b), (c, d)) = (( 1, 2), (1, 2)) $ amounts to checking that this square is bicartesian. We do this in the following lemma.

\begin{lemma}
The Beck-Chevalley square is bicartesian.
\end{lemma}

\begin{proof}
Since the square is commutative, it only remains to check that it is a pullback square of vector spaces. We have the decompositions as free modules
\begin{align*}
\NH_3 &= III \NH_{1, 2} ⊕ XI \NH_{1, 2} ⊕ W \NH_{1, 2}, \\
\NH_{2, 1} &= III \NH_{1, 1, 1} ⊕ XI \NH_{1, 1, 1}, \\
\NH_{1, 2} &= III \NH_{1, 1, 1} ⊕ IX \NH_{1, 1, 1}.
\end{align*}
With the help of these decompositions we can identify an element of each space in the square as sending any of the elementary crossing diagrams $ III, IX, XI, XI · IX, W, W · IX $ to an element of $ T $. In terms of these, the diagram reads
\begin{equation*}
\begin{tikzcd}
T ¤ \vspan(III, XI, W)^{∨} \arrow[r] \arrow[d] & T ¤ \vspan(III, IX, XI, W)^{∨} \arrow[d] \\
T ¤ \vspan(III)^{∨} \arrow[r] & T ¤ \vspan(III, IX)^{∨}
\end{tikzcd}
\end{equation*}
The left map is the restriction map, sending $ A III^{∨} + B XI^{∨} + C W^{∨} $ to $ A III^{∨} $. The bottom map sends $ A III^{∨} $ to the function $ φ $ with $ φ(III) = A $ and $ φ(IX) = A · IX $ thus $ φ = A · III^{∨} + A· IX · IX^{∨} $. The top map sends $ f = A III^{∨} + B XI^{∨} + C W^{∨} $ to the double-function $ φ $ determined by $ φ(IX)(XI) = f(IX · XI) = f(W) = C $ and $ φ(III)(III) = f(III) = A $ and $ φ(III)(XI) = f(XI) = B $ and $ φ(IX)(III) = f(IX) = A · IX $, thus to $ φ = A · III^{∨} + (A · IX) · IX^{∨} + B · XI^{∨} + C W^{∨} $. Expressed as matrices, the top map reads $ (A, B, C) ↦ (A, A · IX, B, C) $. The right map is the restriction map to the first two components. It is now obvious that the diagram is bicartesian.
\end{proof}

\subsection{The bifactorization cube $ Q_{\NHS_5} ((2, 3), (2, 3)) $}
\label{sec:example-Q_23_23}
In this section, we demonstrate the iterative procedure of \autoref{th:nh-ccmccm-vanishing}. We focus on the case of $ Q_{\NHS_5} ((2, 3), (2, 3)) $ and illustrating the diagram sets which appear during the procedure. In this specific case we have $ c = 2 $ and $ m = 1 $.

The iteration proceeds in three steps, starting with the Beck-Chevalley cube $ B^3 $ and ending with $ B^0 $. For every step during the iteration, we illustrate the diagrams that characterize the vertices of $ B^{3-i} $. Instead of the individual diagram sets $ \Anycross^i_τ $ and $ \Mincross^i_τ $, we shall depict only the compositions of the diagrams from these two sets. This makes it easy to trace how the diagrams of $ B^2 $ arise from those of $ B^3 $ and how the diagrams of $ B^1 $ arise from those of $ B^2 $.

For notational ease, we write $ σ $ immediately instead of $ \Perf \NH_σ $ in the Beck-Chevalley and intermediate Beck-Chevalley cubes. Note that functor composition direction is from top to bottom following \autoref{conv:nh-algebra-strandconv}, while the strand reading direction is from bottom to top, following \autoref{conv:prelim-bc-functorconv}. The reading directions are illustrated in \autoref{fig:example-Q2323-000_001}.

\input{example/Q_23_23/fig_cells.tex}

\paragraph*{The Beck-Chevalley cube $ B^3 $.}
The Beck-Chevalley cube $ B^3 = \BC(Q_{\NHS_5} ((2, 3), (2, 3))) $ is 3-dimensional. Its axes are labeled by the two binary digits $ β_1, β_2 $ and a third digit which distinguishes top and bottom layer. The 8 vertices are depicted in \autoref{fig:example-Q2323-000_001} till \ref{fig:example-Q2323-110_111}. The top layer is given by the vertices
\begin{equation*}
\begin{tikzcd}
B^3_{0,0,0} \arrow[r] \arrow[d] & B^3_{1,0,0} \arrow[d] \\
B^3_{0,1,0} \arrow[r] & B^3_{1,1,0}
\end{tikzcd}
\quad \text{with dimensions}
\begin{tikzcd}
10 \arrow[r] \arrow[d] & 12 \arrow[d] \\
18 \arrow[r] & 24.
\end{tikzcd}
\end{equation*}
The bottom layer is given by the vertices
\begin{equation*}
\begin{tikzcd}
B^3_{0,0,1} \arrow[r] \arrow[d] & B^3_{1,0,1} \arrow[d] \\
B^3_{0,1,1} \arrow[r] & B^3_{1,1,1}
\end{tikzcd}
\quad \text{with dimensions}
\begin{tikzcd}
1 \arrow[r] \arrow[d] & 6 \arrow[d] \\
3 \arrow[r] & 12.
\end{tikzcd}
\end{equation*}

\paragraph*{The intermediate Beck-Chevalley cube $ B^2 $.}
The intermediate cube $ B^2 $ is 2-dimensional. Its axes are labeled by the two binary digits $ β_1, β_2 $. The 4 vertices are the kernels $ B^2_{β_1, β_2} = \Ker(B^3_{β_1, β_2, 0} → B^3_{β_1, β_2, 1}) $ of the vertical maps from top to bottom layer. The vertices are depicted in \autoref{fig:example-Q2323-00_10} and \ref{fig:example-Q2323-01_11}. The intermediate cube is
\begin{equation*}
\begin{tikzcd}
B^2_{0,0} \arrow[r] \arrow[d] & B^2_{1,0} \arrow[d] \\
B^2_{0,1} \arrow[r] & B^2_{1,1}
\end{tikzcd}
\quad \text{with dimensions}
\begin{tikzcd}
9 \arrow[r] \arrow[d] & 6 \arrow[d] \\
15 \arrow[r] & 12.
\end{tikzcd}
\end{equation*}

\paragraph*{The intermediate Beck-Chevalley cube $ B^1 $.}
The intermediate cube $ B^1 $ is 1-dimensional. Its single axis is labeled by the binary digit $ β_2 $. The 2 vertices are the kernels $ B^1_{β_2} = \Ker(B^2_{0, β_2} → B^2_{1, β_2}) $. They are depicted in \autoref{fig:example-Q2323-0_1}. The intermediate cube is
\begin{equation*}
\begin{tikzcd}
B^1_0 \arrow[r]  & B^2_1
\end{tikzcd}
\quad \text{with dimensions}
\begin{tikzcd}
3 \arrow[r] & 3.
\end{tikzcd}
\end{equation*}

\paragraph*{Defect vanishing.}
The total fiber of the Beck-Chevalley cube is the fiber of $ B^1 $. Regard the source composition $ (2, 3) $ and the target composition $ (2, 3) $. We observe that the 3 diagrams in both vertices of $ B^1 $ are precisely the $ (2, 3) $-shuffles which send the strands of the source $ 2 $ block into the target $ 3 $ block and the leftmost $ 2 $ strands of the source $ 3 $ block into the target $ 2 $ block. This precisely confirms the crossing behavior formalized in \autoref{sec:nh-ccmccm}. We observe that the three diagrams in the two vertices of $ B^1 $ are precisely the same. This means that the fiber of $ B^1 $ vanishes, which proves the Defect vanishing axiom.

\section{Remaining checks of twist invertibility and defect vanishing}
\label{sec:appchecks} 
In this section, we provide remaining checks of the twist invertibility and defect vanishing axioms. We start with a pair of compositions $ (a, b), (c, d) ∈ \Comp(n+1) $. We abide by the case distiction provided in \autoref{sec:prelim-bifcube} but skip the case $ a < c $ because of the apparent symmetry in the bifactorization cubes. We therefore assume $ a ≥ c $ and distinguish six cases. We provide a slightly more elaborate description for the case of $ (a, a), (a, a) $ but generally focus on providing very brief descriptions of the bifactorization cube and its associated BC cube so that it becomes clear that these are analogous to the case $ ((c, c+m), (c, c+m)) $.

\paragraph*{Twist invertibility for $ Q_{\NHS_{2a}} ((a, a), (a, a)) $.}
This bifactorization cube is $ a+1 $-dimensional and we have
\begin{equation*}
Q_{\NHS_{2a}} ((a, a), (a, a)))_{δ_1, δ_2, β_1, …, β_{a-1}} = \Perf \NH_{ψ^{-1} (ε_1 … ε_{a-1} (δ_1 ∨ δ_2) ε_{a-1} … ε_1)}.
\end{equation*}
Thus the Beck-Chevalley cube is $ a $-dimensional and given by
\begin{align*}
\BC(Q_{\NHS_{2a}} ((a, a), (a, a)))_{β_1, …, β_{a-1}, 0} &=
\begin{array}[c]{l}
\Perf \NH_{ψ^{-1} (0^{c-1} 1 0^{c-1})} \\ \Perf \NH_{ψ^{-1} (β_1 … β_{c-1} 1 β_{c-1} … β_1)} \\ \Perf \NH_{ψ^{-1} (β_1 … β_{c-1} 0 β_{c-1} … β_1)} \\ \Perf \NH_{ψ^{-1} (β_1 … β_{c-1} 1 β_{c-1} … β_1)} \\ \Perf \NH_{ψ^{-1} (0^{c-1} 1 0^{c-1})}
\end{array}, \\
\BC(Q_{\NHS_{2a}} ((a, a), (a, a)))_{β_1, …, β_{a-1}, 1} &=
\begin{array}[c]{l}
\Perf \NH_{ψ^{-1} (0^{c-1} 1 0^{c-1})} \\ \Perf \NH_{ψ^{-1} (β_1 … β_{c-1} 1 β_{c-1} … β_1)} \\ \Perf \NH_{ψ^{-1} (β_1 … β_{c-1} 1 β_{c-1} … β_1)} \\ \Perf \NH_{ψ^{-1} (β_1 … β_{c-1} 1 β_{c-1} … β_1)} \\ \Perf \NH_{ψ^{-1} (0^{c-1} 1 0^{c-1})}.
\end{array}
\end{align*}
Let us now explain how to compute the total fiber of the Beck-Chevalley cube. We deploy a similar iterative procedure as in \autoref{sec:nh-ccmccm}. In every step, the vertices of the intermediate Beck-Chevalley cubes can be described by means of maps from a set of diagrams to a given $ T ∈ \Perf \NH_{a, a} $. The minimum number of crossings of these diagrams is equal to $ i $ after the $ i $-th iteration of the process.

After the $ a $-th iteration, the intermediate Beck-Chevalley cube consists of a single vertex. The single diagram $ X $ associated to this vertex is the total crossing on $ 2a $ strands between the first $ a $ and the second $ a $ strands. It is depicted in \autoref{fig:appcheck-aaaa-crossing}. In conclusion, the total fiber of the Beck-Chevalley cube is given by $ \Map(\{X\}, T) $. In other words, the total fiber is a copy of $ T $ itself. It inherits an $ \NH_{a, a} $-module structure which is given by $ t.n = tφ(n) $ where $ φ: \NH_{a, a} = \NH_a ¤ \NH_a → \NH_{a, a} = \NH_a ¤ \NH_a $ is the tensor flip. This means that the total fiber is simply the flip map $ \Perf \NH_{a, a} → \Perf \NH_{a, a} $, and finishes the proof of twist invertibility.

\begin{figure}
\centering
\begin{tikzpicture}
\shufflebox{(0, 0)}{8}{A1}{0.125};
\path (A1-C) node {$ a+a $};
\shufflebox{(0, -1)}{4}{B1}{0};
\shufflebox{(2.25, -1)}{4}{B2}{0};
\path (B1-C) node {$ a $};
\path (B2-C) node {$ a $};
\path[draw] (B1-T1) -- (A1-B5);
\path[draw] (B1-T2) -- (A1-B6);
\path[draw] (B1-T3) -- (A1-B7);
\path[draw] (B1-T4) -- (A1-B8);
\path[draw] (B2-T1) -- (A1-B1);
\path[draw] (B2-T2) -- (A1-B2);
\path[draw] (B2-T3) -- (A1-B3);
\path[draw] (B2-T4) -- (A1-B4);
\end{tikzpicture}
\caption{This figure depicts the total crossing between two blocks of $ a $ strands each. This diagram is the only remaining diagram after the iterative procedure for computing the total fiber of the Beck-Chevalley cube $ \BC(Q_{\NHS_{2a}} ((a, a), (a, a))) $.}
\label{fig:appcheck-aaaa-crossing}
\end{figure}
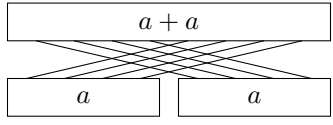

\paragraph*{Twist invertibility for $ Q_{\NHS_{2c+l}} ((c+l, c), (c, c+l)) $.}
This bifactorization cube is $ c+1 $-dimensional and
\begin{equation*}
Q_{\NHS_{2c+l}} ((c+l, c), (c, c+l))_{δ_1, δ_2, ε_1, …, ε_{c-1}} = \Perf \NH_{ψ^{-1} (ε_1 … ε_{c-1} (δ_1 0^{l-1} δ_2) ε_{c-1} … ε_1)}.
\end{equation*}
The cube $ \BC(Q_{\NHS_{2c+l}} ((c+l, c), (c, c+l))) $ is $ c $-dimensional and
\begin{align*}
\BC(Q_{\NHS_{2c+l}} ((c+l, c), (c, c+l)))_{β_1, …, β_{c+m-1}, 0} &= 
\begin{array}[c]{l}
\Perf \NH_{ψ^{-1} (0^{c-1} 00^{l-1}1 0^{c-1})} \\ \Perf \NH_{ψ^{-1} (β_1 … β_{c-1} 00^{l-1}1 β_{c-1} … β_1)} \\ \Perf \NH_{ψ^{-1} (β_1 … β_{c-1} 00^{l-1}0 β_{c-1} … β_1)} \\ \Perf \NH_{ψ^{-1} (β_1 … β_{c-1} 10^{l-1}0 β_{c-1} … β_1)} \\ \Perf \NH_{ψ^{-1} (0^{c-1} 10^{l-1}0 0^{c-1})}
\end{array}, \\
\BC(Q_{\NHS_{2c+l}} ((c+l, c), (c, c+l)))_{β_1, …, β_{c+m-1}, 1} &= 
\begin{array}[c]{l}
\Perf \NH_{ψ^{-1} (0^{c-1} 00^{l-1}1 0^{c-1})} \\ \Perf \NH_{ψ^{-1} (β_1 … β_{c-1} 00^{l-1}1 β_{c-1} … β_1)} \\ \Perf \NH_{ψ^{-1} (β_1 … β_{c-1} 10^{l-1}1 β_{c-1} … β_1)} \\ \Perf \NH_{ψ^{-1} (β_1 … β_{c-1} 10^{l-1}0 β_{c-1} … β_1)} \\ \Perf \NH_{ψ^{-1} (0^{c-1} 10^{l-1}0 0^{c-1})}.
\end{array}
\end{align*}

The iterative procedure works in a way analogous to the case of $ ((a, a), (a, a)) $. The final intermediate Beck-Chevalley cube is $ \BC(Q_{\NHS_{2c+l}} ((c+l, c), (c, c+l)))^0 (T) = \Map(\{X\}, T) $. The single permutation $ X $ shuffles the first $ c+l $ strands across the remaining $ c $ strands. This establishes twist invertibility.

\paragraph*{Defect-vanishing for $ Q_{\NHS_{2b+m}} ((b+m, b), (b+m, b)) $.}
Let $ b, m > 0 $. The cube $ Q_{\NHS_{2b+m}} ((b+m, b), (b+m, b)) $ is $ b+2 $-dimensional and
\begin{equation*}
Q_{\NHS_{2b+m}} ((b+m, b), (b+m, b))_{δ_1, δ_2, ζ, ε_1, …, ε_{b-1}} = \Perf \NH_{ψ^{-1} (0^{m-1} ζ ε_1 … ε_{c-1} (δ_1 ∨ δ_2) ε_{c-1} … ε_1)}.
\end{equation*}
Its Beck-Chevalley cube is $ b+1 $-dimensional and
\begin{align*}
\BC(Q_{\NHS_{2b+m}} ((b+m, b), (b+m, b)))_{ζ, β_1, …, β_{b-1}, 0} &= 
\begin{array}[c]{l}
\Perf \NH_{ψ^{-1} (0^{m-1} 0 0^{c-1} 1 0^{c-1})} \\ \Perf \NH_{ψ^{-1} (0^{m-1} ζ β_1 … β_{c-1} 1 β_{c-1} … β_1)} \\ \Perf \NH_{ψ^{-1} (0^{m-1} ζ β_1 … β_{c-1} 0 β_{c-1} … β_1)} \\ \Perf \NH_{ψ^{-1} (0^{m-1} ζ β_1 … β_{c-1} 1 β_{c-1} … β_1)} \\ \Perf \NH_{ψ^{-1} (0^{m-1} 0 0^{c-1} 1 0^{c-1})}
\end{array}, \\
\BC(Q_{\NHS_{2b+m}} ((b+m, b), (b+m, b)))_{ζ, β_1, …, β_{b-1}, 1} &=
\begin{array}[c]{l}
\Perf \NH_{ψ^{-1} (0^{m-1} 0 0^{c-1} 1 0^{c-1})} \\ \Perf \NH_{ψ^{-1} (0^{m-1} ζ β_1 … β_{c-1} 1 β_{c-1} … β_1)} \\ \Perf \NH_{ψ^{-1} (0^{m-1} ζ β_1 … β_{c-1} 1 β_{c-1} … β_1)} \\ \Perf \NH_{ψ^{-1} (0^{m-1} ζ β_1 … β_{c-1} 1 β_{c-1} … β_1)} \\ \Perf \NH_{ψ^{-1} (0^{m-1} 0 0^{c-1} 1 0^{c-1})}.
\end{array}
\end{align*}
The situation is exactly analogous to $ Q_{\NHS_{2c+m}} ((c, c+m), (c, c+m)) $. We finish the proof.

\paragraph*{Defect-vanishing for $ Q_{\NHS_{2b+m+l}} ((b+m+l, b), (b+m, b+l)) $.}
Let $ b, m, l > 0 $. The cube $ Q_{\NHS_{2b+m+l}} ((b+m+l, b), (b+m, b+l)) $ is $ b+2 $-dimensional and
\begin{equation*}
Q_{\NHS_{2b+m+l}} ((b+m+l, b), (b+m, b+l))_{δ_1, δ_2, ζ, ε_1, …, ε_{b-1}} = \Perf \NH_{ψ^{-1} (0^{m-1} ζ ε_1 … ε_{b-1} δ_1 0^{l-1} δ_2 ε_{c-1} … ε_1)}.
\end{equation*}
Its Beck-Chevalley cube is $ b+1 $-dimensional and
\begin{align*}
\BC(Q_{\NHS_{2b+m+l}} ((b+m+l, b), (b+m, b+l)))_{ζ, β_1, …, β_{b-1}, 0} &= 
\begin{array}[c]{l}
\Perf \NH_{ψ^{-1} (0^{m-1} 0 0^{b-1} 00^{l-1}1 0^{b-1})} \\\Perf \NH_{ψ^{-1} (0^{m-1} ζ β_1 … β_{b-1} 00^{l-1}1 β_{b-1} … β_1)} \\ \Perf \NH_{ψ^{-1} (0^{m-1} ζ β_1 … β_{b-1} 00^{l-1}0 β_{b-1} … β_1)} \\ \Perf \NH_{ψ^{-1} (0^{m-1} ζ β_1 … β_{b-1} 10^{l-1}0 β_{b-1} … β_1)} \\ \Perf \NH_{ψ^{-1} (0^{m-1} 0 0^{b-1} 10^{l-1}0 0^{b-1})}
\end{array}, \\
\BC(Q_{\NHS_{2b+m+l}} ((b+m+l, b), (b+m, b+l)))_{ζ, β_1, …, β_{b-1}, 1} &=
\begin{array}[c]{l}
\Perf \NH_{ψ^{-1} (0^{m-1} 0 0^{b-1} 00^{l-1}1 0^{b-1})} \\ \Perf \NH_{ψ^{-1} (0^{m-1} ζ β_1 … β_{b-1} 00^{l-1}1 β_{b-1} … β_1)} \\ \Perf \NH_{ψ^{-1} (0^{m-1} ζ β_1 … β_{b-1} 10^{l-1}1 β_{b-1} … β_1)} \\ \Perf \NH_{ψ^{-1} (0^{m-1} ζ β_1 … β_{b-1} 10^{l-1}0 β_{b-1} … β_1)} \\ \Perf \NH_{ψ^{-1} (0^{m-1} 0 0^{b-1} 10^{l-1}0 0^{b-1})}.
\end{array}
\end{align*}
The situation is exactly analogous to $ Q_{\NHS_{2c+m}} ((c, c+m), (c, c+m)) $. We finish the proof.

\paragraph*{Defect vanishing for $ Q_{\NHS_{2c+l+m}} ((c+l, c+m), (c, c+m+l)) $.}
Let $ c, l, m > 0 $. The cube $ Q_{\NHS_{2c+l+m}} ((c+l, c+m), (c, c+m+l)) $ is $ c+2 $-dimensional and
\begin{equation*}
Q_{\NHS_{2c+l+m}} ((c+l, c+m), (c, c+m+l))_{δ_1, δ_2, β_1, …, β_{c-1}, ζ} = \Perf \NH_{ψ^{-1} (ε_1 … ε_{c-1} δ_1 0^{l-1} δ_2 ε_{c-1} … ε_1 ζ 0^{m-1})}.
\end{equation*}
Its Beck-Chevalley cube is $ c+1 $-dimensional and
\begin{align*}
\BC(Q_{\NHS_{2c+l+m}} ((c+l, c+m), (c, c+m+l)))_{ζ, β_1, …, β_{c-1}, 0} &=
\begin{array}[c]{l}
\Perf \NH_{ψ^{-1} (0^{c-1} 00^{l-1}1 0^{c-1} ζ 0^{m-1})} \\ \Perf \NH_{ψ^{-1} (β_1 … β_{c-1} 00^{l-1}1 β_{c-1} … β_1 ζ 0^{m-1})} \\ \Perf \NH_{ψ^{-1} (β_1 … β_{c-1} 00^{l-1}0 β_{c-1} … β_1 ζ 0^{m-1})} \\ \Perf \NH_{ψ^{-1} (β_1 … β_{c-1} 10^{l-1}0 β_{c-1} … β_1 ζ 0^{m-1})} \\ \Perf \NH_{ψ^{-1} (0^{c-1} 10^{l-1}0 0^{c-1} 0 0^{m-1})}
\end{array}, \\
\BC(Q_{\NHS_{2c+l+m}} ((c+l, c+m), (c, c+m+l)))_{ζ, β_1, …, β_{c-1}, 1} &= 
\begin{array}[c]{l}
\Perf \NH_{ψ^{-1} (0^{c-1} 00^{l-1}1 0^{c-1} ζ 0^{m-1})} \\ \Perf \NH_{ψ^{-1} (β_1 … β_{c-1} 00^{l-1}1 β_{c-1} … β_1 ζ 0^{m-1})} \\ \Perf \NH_{ψ^{-1} (β_1 … β_{c-1} 10^{l-1}1 β_{c-1} … β_1 ζ 0^{m-1})} \\ \Perf \NH_{ψ^{-1} (β_1 … β_{c-1} 10^{l-1}0 β_{c-1} … β_1 ζ 0^{m-1})} \\ \Perf \NH_{ψ^{-1} (0^{c-1} 10^{l-1}0 0^{c-1} 0 0^{m-1})}.
\end{array}
\end{align*}
Evaluating the BC cube of this bifactorization cube is analogous to the case of $ Q_{\NHS_{2c+m}} ((c, c+m), (c, c+m)) $ with a slight amount of complexity added through the fact that the cube reduces to $ Q_{\NHS_{2c+l}} ((c+l, c), (c, c+l)) $ instead of $ Q_{\NHS_{2c}} ((c, c), (c, c)) $.

\printbibliography

\bigskip

\begin{tabular}{@{}l@{}}%
    \textsc{Department of mathematics, University of California, Berkeley, USA}\\
    \textit{jasper.kreeke@berkeley.edu}
  \end{tabular}

\end{document}